\documentclass[preprint,12pt]{elsarticle}

\usepackage{amssymb}
\usepackage{amsmath}
 \usepackage{amsthm}
\RequirePackage[numbers]{natbib}
\RequirePackage[colorlinks,citecolor=blue,urlcolor=blue]{hyperref}
\RequirePackage{graphicx}
\usepackage{amsfonts}
\usepackage{hyperref}
\usepackage{amsmath}
\usepackage{xcolor}
\usepackage{amsthm}
\usepackage{pdflscape}
\usepackage{pgfplots}
\pgfplotsset{compat=1.18}
\usepackage{mathrsfs}
\usepackage{lettrine}
\usepackage{scrextend}
\usepackage{enumitem}
\usepackage{bbm}
\usepackage{mathtools}
\usepackage{graphicx}
\usepackage{amssymb}
\usepackage{mwe}
\usepackage[font=small]{caption}
\usepackage{subcaption}  

\usepackage{tikz}
\usetikzlibrary{positioning}
\tikzset{
  dend/.style={line width=0.6pt},
  leaf/.style={circle, fill=black, inner sep=0pt, minimum size=2.2pt},
}

\theoremstyle{plain}

\newtheorem{remark}{Remark}

\newtheorem{theorem}{Theorem}[section]
\newtheorem{lemma}[theorem]{Lemma}
\newtheorem{proposition}[theorem]{Proposition}
\newtheorem{corollary}[theorem]{Corollary}
\theoremstyle{remark}
\newtheorem{definition}[theorem]{Definition}
\newtheorem{example}{Example}

\usepackage{etoolbox}

\apptocmd{\thebibliography}{%
\setlength{\itemsep}{0pt}%
\setlength{\parskip}{0pt}%
}{}{}

\begin{document}

\begin{frontmatter}



\title{Hierarchical Clustering Algorithms on Poisson\\ and Other Stationary Point Processes}


\author[]{Sayeh Khaniha\corref{cor1}}

\author[]{François Baccelli}



\begin{abstract}
This paper introduces a hierarchical clustering algorithm, the Clustroid Hierarchical Nearest Neighbor (\(\mathrm{CHN}^2\)),
designed for datasets with a countably infinite number of data represented by points in the Euclidean space. The method builds clusters across successive levels
by linking nearest-neighbor points or clusters using the clustroid distance. The properties of this algorithm make
it suitable for very large datasets.

To evaluate its properties, we first apply the algorithm to the homogeneous Poisson point process,
which serves as a natural null-hypothesis model with no intrinsic data aggregation. In this setting, the 
algorithm generates a random forest that is a deterministic factor of the Poisson point process,
and hence unimodular. We prove that at every level, the level-\(k\) graph has only finite connected
components (a.s.) and derive bounds on their mean size. We also establish the existence of a limiting
graph as the number of levels tends to infinity. In this limit, clusters are shown to be all infinite and one-ended,
which induces a natural order within each component and supports a tree-like phylogenetic interpretation.

Beyond the Poisson case, we extend the analysis to a class of Cox and more general stationary point
processes without second-order descending chains (introduced here), for which analogous results hold. Simulations
show that comparing the Cox case with the Poisson baseline allows an efficient detection of aggregation,
thereby linking the stochastic-geometric analysis to practical clustering tasks.
\end{abstract}




\end{frontmatter}


\section{Introduction}\label{INTRO}

Clustering is a fundamental tool in unsupervised learning whose objective is to identify groups of similar objects within a dataset. Among clustering methods, hierarchical clustering occupies a special place because it reveals the organization of the data at multiple scales, producing a nested sequence of clusters rather than a single partition. Classical hierarchical clustering algorithms have been extensively studied and successfully applied in a wide range of domains; see \cite{Murtagh,Murtagh2} and the references therein. However, most existing approaches are designed for finite datasets and are typically analyzed in a deterministic setting.

In many applications, data are naturally represented as realizations of spatial point processes rather than finite point clouds. Examples arise in telecommunications, biology, astronomy, and materials science, where the number of observations can be extremely large, and the underlying data-generating mechanism is inherently random. This motivates the study of hierarchical clustering procedures that can be defined directly on random point configurations and whose structural properties can be analyzed within a probabilistic framework. In this paper, we introduce a hierarchical clustering algorithm on stationary point processes and investigate its geometric and probabilistic properties, using the homogeneous Poisson point process (PPP) as a baseline model.

\paragraph{The model} We now describe the algorithm informally; the precise construction is given in Subsection~\ref{0011C}. Let $\phi$ be a locally finite point configuration in $\mathbb{R}^d$. At level $0$, each point of $\phi$ is connected to its nearest neighbor. The connected components of the resulting nearest-neighbor graph are called \textit{clusters of order $0$}. Whenever such a cluster is finite, and distances are in general position (no distance ties), it contains a unique mutual nearest-neighbor pair. This pair is used as the representative, or \emph{clustroid pair}, of the cluster.

Higher levels are constructed recursively. Suppose that the clusters of order $k$ are finite and each has a well-defined clustroid pair. The distance between two clusters is defined through their clustroid pairs using the single-linkage distance introduced in Subection~\ref{0011C}. Each cluster is then connected to its nearest neighboring cluster with respect to this distance, and the connected components of the resulting graph are called clusters of order $k+1$. Assuming that the clusters of order \(k+1\) are finite, each such cluster contains a unique mutual nearest-neighbor pair. These pairs define the clustroid pairs of order (k+1), and the procedure is iterated. We call this algorithm the \emph{Clustroid Hierarchical Nearest Neighbor} algorithm, denoted by $\mathrm{CHN}^2$, where the superscript “$2$” is shorthand for “NN”.

A simple phylogeny picture is useful here. One can view clustroid pairs as archetypes of their components.
Linking archetypes across levels organizes the data into a tree-like summary.

The clustering algorithm scales well to large data. Once a component forms, only its two clustroid points
are needed to decide future links. There is no extra optimization step to \emph{find} representatives since
the clustroid pair arises automatically as an MNN pair inside the component. This autonomy reduces computation
and memory, since decisions are made on representatives rather than on all points, and it also enables
parallelization: widely separated regions can be processed independently and hence in parallel.



To study properties at the scale of very large or infinite datasets, we place the algorithm in a point-process framework. As a baseline, we consider the homogeneous Poisson point process
\begin{equation}
\label{PPPdefinition}
\Phi^0=\sum_i\delta_{x_i}
\end{equation}
in \(\mathbb R^d\), which has no intrinsic aggregation (see Remark \ref{aggragates}) and therefore serves as a natural null model. The construction is scale invariant, so there is no loss of generality in assuming unit intensity. A key observation is that each level $k$ of the hierarchy is a random directed graph that can be represented by a point-shift $f_{k}$ on the underlying point process. A \emph{point-shift} $F$ maps, in a translation invariant way, each point of a stationary point process $\phi$ to some point of $\phi$ (see \cite{Mecke}, \cite{Tho00}, and \cite{Pointshift} for more details and Section \ref{Ch:apter7}).  Consequently, the level graphs are \emph{unimodular} (so the mass-transport arguments apply; see \cite{Aldous} for more details) under the Palm distribution, allowing us to use tools from the theory of unimodular random networks. 

\paragraph{Questions} The recursive nature of the construction raises several questions. Are the finite-level point-shift graphs $(f_k)$ well defined, i.e., are all clusters almost surely finite at every level? Can the hierarchy be iterated indefinitely, and does the sequence $(f_k)$ converge to a limiting graph? If such a limit exists, what are its structural properties, such as the number of connected components and the number of ends of its infinite clusters? More generally, to what extent do these properties extend beyond the Poisson setting, and can the resulting hierarchy be used to detect aggregation in spatial data?

\paragraph{Main Results} Theorem \ref{FFPROPERTY} states that, at every level, all clusters are almost surely finite; we derive bounds on their mean sizes. This property is proven in Section \ref{ProofFF}. The main new methodological tool is the notion \emph{second-order descending chain} introduced here, see Definition \ref{017}. We show that the PPP admits almost surely no infinite second-order descending chain. This is instrumental in proving the finiteness of the level-\(k\) clusters and yields intensity bounds for the level-\(k\) clustroids. 

Theorem~\ref{Existence of the limit point-shift} establishes the existence of a limiting graph as the number of levels tends to infinity and shows that this limiting graph is unimodular. Unimodularity is important because it turns global properties into local expectations and gives structural leverage: for instance, in a unimodular random graph, the Foil Classification Theorem (see page 10 in \cite{Baccelli2017}) classifies the size and the number of ends in the clusters. It is shown that the mean size of the \emph{typical} cluster in this limiting graph is infinite, and Corollary~\ref{IIconnected} shows that all infinite clusters of the limiting graph are \emph{one-ended.} This property provides a natural genealogical structure within each cluster: any two vertices eventually merge into a common ancestral branch, and the level at which this merger occurs may be interpreted as a hierarchical distance between them. In contrast, multi-ended clusters admit several competing directions to infinity and do not support such a canonical ordering.

It is worth emphasizing that these finiteness and structural results are far from trivial in the infinite setting. 
In finite datasets, clusters produced by hierarchical algorithms are necessarily finite, but when the data form a stationary point process, the situation is subtler. 
Even the nearest-neighbor graph (the first level of the hierarchy) may contain infinite components; this occurs, for example, in certain non-Poisson stationary processes that admit infinite descending chains (see \cite{Daley}). 
The proof that the homogeneous Poisson process yields only finite clusters thus requires specific geometric arguments and plays a key role in establishing the well-definedness of the hierarchy at all levels. 
Likewise, in the limiting infinite graph, the fact that all connected components are \emph{one-ended} is not automatic: graphs with multiple ends can exhibit branching structures where no consistent order can be defined between distant points. The one-endedness property, therefore, ensures that, even in the infinite-data regime, each cluster admits a natural hierarchical order—providing the basis for its phylogenetic interpretation.

The analysis extends beyond the Poisson setting. Theorem \ref{thm:transfer-no2dc} shows that the arguments developed in the PPP case rely primarily on stationarity and the absence of infinite second-order descending chains. In Section~\ref{CoxSection}, we introduce a class of Cox point processes satisfying these properties and show that the finiteness results established for the PPP continue to hold. Simulations indicate that, relative to the Poisson baseline, $(\mathrm{CHN}^2)$ effectively detects aggregation and distinguishes structured organization from purely stochastic dispersion.

The paper is organized as follows. Section~\ref{Ch:apter7} introduces the $\mathrm{CHN}^2$ algorithm on the homogeneous Poisson point process and describes its level-by-level construction through a sequence of point-shifts. It also establishes the scale invariance of the construction and introduces the limiting graph, referred to as the $\mathrm{CHN}^2$ Eternal Family Forest (EFF). Section~\ref{CoxSection} extends the framework to Cox and other stationary point processes and illustrates its use for detecting aggregation through numerical simulations. Section~\ref{ProofFF} contains the proofs of the finiteness of all finite-level clusters, the absence of infinite second-order descending chains, and the main structural results for the hierarchy. Finally, Section~\ref{OpenProblems} discusses open problems and conjectures concerning the connectivity and ergodic properties of the limiting graph.

\section{\texorpdfstring{$\mathrm{CHN}^2$}{CHN2} on Point Processes}
\label{Ch:apter7}
This section studies the $\mathrm{CHN}^2$ algorithm on homogeneous PPP,  \( \Phi^0 \). As explained in the introduction, each level of the hierarchy
can be represented as a point-shift and therefore under the Palm probability is a unimodular graph.
We briefly recall the required notions.
\paragraph{Point-shifts and Unimodularity}
A point-shift $F$ on a stationary point process $\phi$ is a translation invariant map  from the set of atoms of the point process to itself.
Given a point-shift \(F\) on \(\phi\), the associated \(F\)-graph is the Euclidean graph whose vertex set is the set of atoms of \(\phi\) and whose edge set is
$\{(x,F(x)):x\in\phi\}$.
For more details see \cite{haji_mirsadeghi_baccelli2018, Pointshift}.  A random rooted graph \([G,o]\) is said to be unimodular if it satisfies the Mass Transport Principle (MTP); see \cite{Aldous}. Under the Palm probability of $\phi$, the $F$-graph of any point-shift is unimodular, see \cite{Baccelli2017}.

\subsection{Construction of the Pre-limit Point-Shifts}
\label{0011C}
We construct a sequence \( f_n \), \( n \ge 0 \), of point-shifts on \( \Phi^0 \) by induction. 

\paragraph{Order 0}
The first point-shift, \( f_0 \), is the nearest-neighbor point-shift on \( \Phi^0 \) 
with respect to the Euclidean distance between points.
That is, for all \(x \in \Phi^0\), we define \( f_0(x) = \mathrm{NN0}(x) \), where 
\(\mathrm{NN0}\) maps a point to its nearest neighbor in \( \Phi^0 \).

It is proved in Section~\ref{ProofFF} that all components of the $(f_0)$-graph are finite. Since each vertex has out-degree one, every connected component contains a unique directed cycle. In the nearest-neighbor graph this cycle is necessarily a mutual nearest-neighbor pair and therefore has length two.

The construction for levels $(k\ge1)$ is recursive. Assume that the point-shift $(f_k)$ has been constructed. By Proposition~\ref{008}, every connected component of the $(f_k)$-graph is almost surely finite and contains a unique cycle of length two. We now introduce the notation associated with these components before defining the next point-shift $(f_{k+1})$.

\begin{definition}
Each connected component of the \( f_k \)-graph on the PPP has a cycle of length two.
These cycles are called \textbf{\(k\)-cycles}. The connected components of the \( f_k \)-graph
are also referred to as \textbf{clusters of order $k$}. The vertices in the \( f_k \)-graph that belong to
\(k\)-cycles are called \textbf{cluster heads of order $k$}.
\end{definition}
\begin{definition}
The union of cluster heads of order $k$ forms a sub-PP of \(\Phi^0\) denoted by \(\Phi^k_c\). Let \(\Phi^k_a = \Phi^0 \setminus \Phi^k_c\). 
If the edges of the \(k\)-cycles in the \( f_k \)-graph are deleted,
one obtains a collection of directed trees referred to as the \textbf{cluster subtrees of order k}.
Each vertex \( x \) in a cluster subtree of order k is either a cluster head of order k
(i.e., \( x \in \Phi^k_c \)) or is such that \( x \in \Phi^k_a \), and there is a path from \( x \)
to a cluster head of order $k$. See Figure \ref{Fi0} and \ref{Fi1} . 
\end{definition}
\begin{figure}
\centering
	\includegraphics[width=6cm]{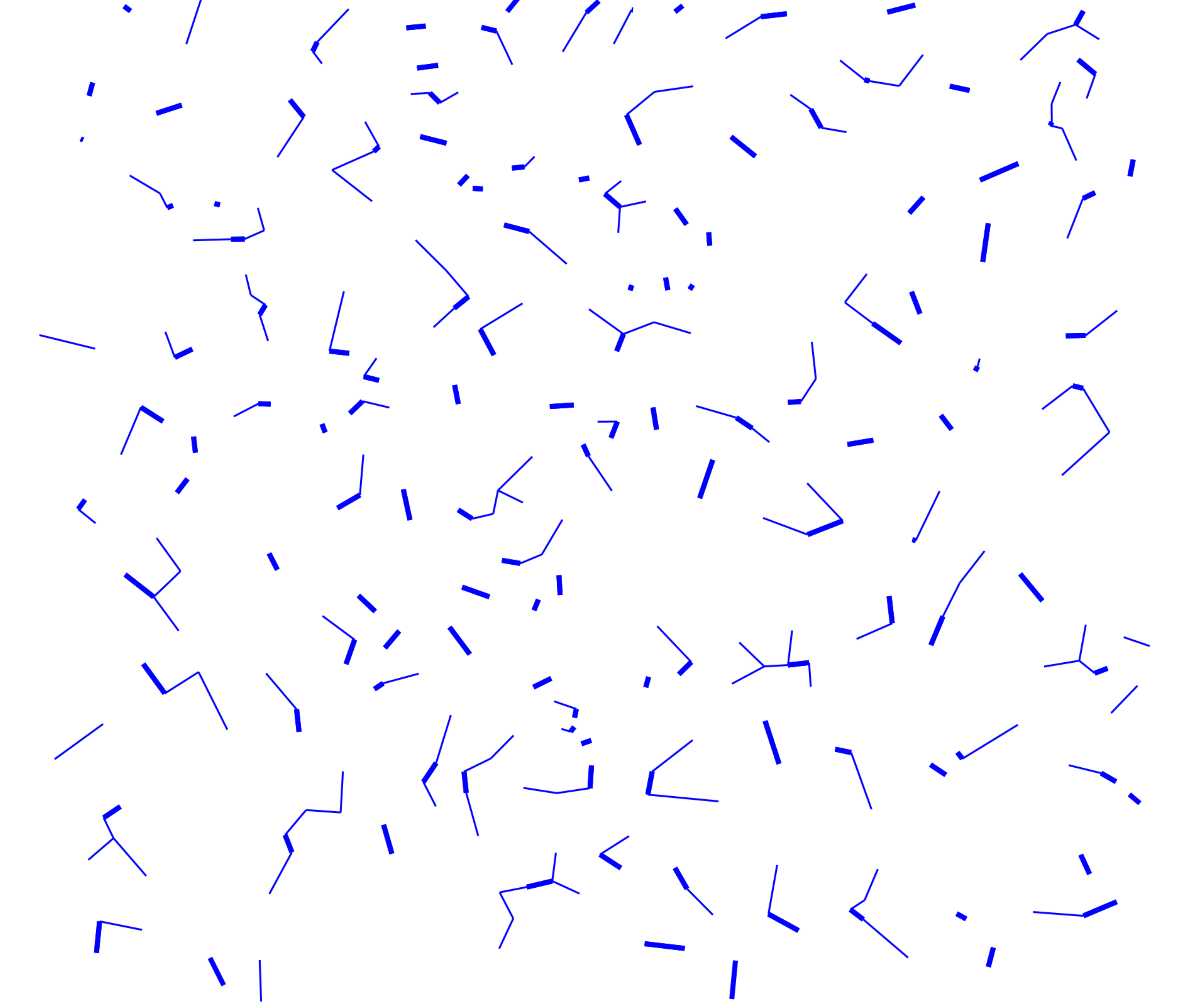} \hspace{1cm}
\includegraphics[width=6cm]{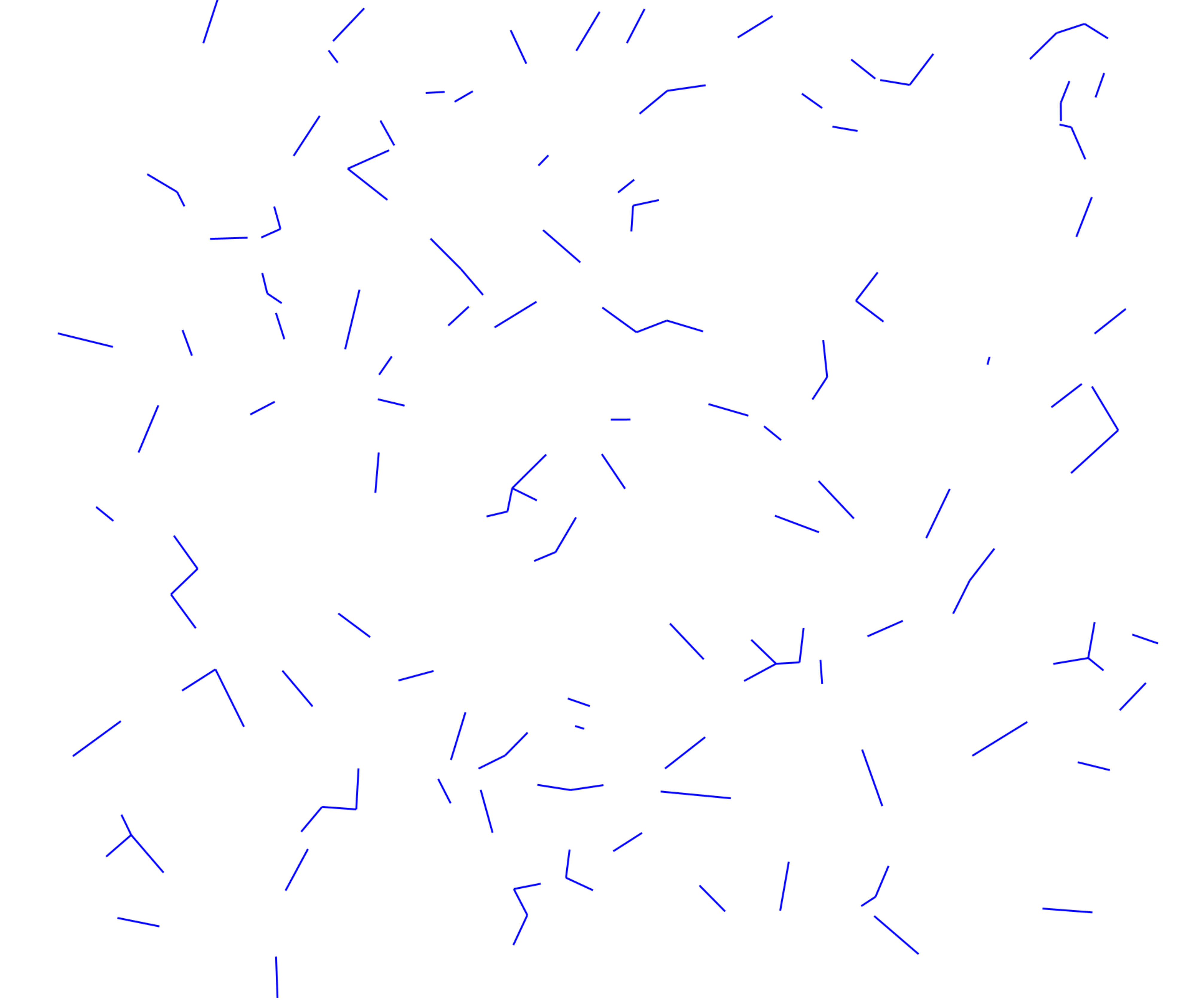}
\caption{\textbf{Left picture}: The \( f_0 \)-graph generated on PPP in blue with the \(0\)-cycles shown in bold;
\textbf{Right picture}: Cluster subtrees of order 0 obtained by deleting the edges of the \(0\)-cycles 
    }
\label{Fi0}
\end{figure}
\begin{definition}
\label{Exit0}
Let \(\{C^k_i\}_{i}\) denote the collection of \(k\)-cycles,
and let \(\{S^k_i\}_i\) be the collection of pairs of points in these cycles,
namely \(S^k_i\) is the pair of head points belonging to the cycle \(C^k_i\).
The sets \(\{S^k_i\}_i\) form a translation-invariant partition of the support of \(\Phi^k_c\).

Consider the nearest neighbor map, denoted by \textbf{NNk}, on \(\{S^k_i\}_i\),
which maps each pair \(S^k_i\) to its nearest neighbor pair in \(\{S^k_j\}_j\). 

The distance between pairs is measured using the single-linkage pseudo-distance
\begin{equation}
\label{eq:myd1}
\delta(S,T)=\min_{x\in S,\;y\in T}\mathfrak d(x,y),
\end{equation}
where $\mathfrak d$ denotes the Euclidean distance.


Specifically, \(S_i^k\) maps to \(S_j^k\) if \(\delta(S_j^k, S_i^k) < \delta(S_k^k, S_i^k)\) for all \(k \ne i, j\).

The \textbf{NNk graph} is a directed graph on \(\{S^k_i\}_i\).
There is an edge from \(S^k_i\) to \(\mathrm{NNk}(S^k_i)\), the nearest neighbor of \(S^k_i\) based on the single linkage distance.
Let \(\mathrm{NNk}(S^k_i) = S^k_j\). There is a point, denoted \(x_{i}\), in \(S^k_i\) and a point, \(x_{j}\), in \(S^k_j\) 
that achieve the minimum in \(\delta(S^k_i, S^k_j)\), denoted by \(\mathrm{nnk}(x_{i}) = x_{j}\).
The point \(x_{i}\) is called the \textbf{exit point of order $k$} in the pair \(S^k_i\).
\end{definition}
Using the exit points defined above, we now construct the next level of the hierarchy. Define \( f_{k+1} \) to be the point-shift on \( \Phi^0 \) that coincides with \( f_k \)
everywhere except for the exit points of order $k$. For each such point, say \( x \in S^k \), where \( S^k \)
is the pair where \( x \) is the exit point, define \( f_{k+1}(x) \) to be the point \(\mathrm{nnk}(x)\).
In other words, for each exit point \( x \) of an \( f_k \)-cycle, \( f_{k+1} \) is obtained from \( f_k \)
by replacing the initial image (which was the mutual nearest neighbor of \( x \) with respect to \( d \))
by \( \mathrm{nnk}(x) \). 

As mentioned before each connected component of the graph of \( f_{k+1} \) on \( \Phi^0 \) are the \textbf{cluster of order $k+1$}.
It will be shown in Corollary \ref{006} and Proposition \ref{008} that these connected components are almost surely finite meaning that each cluster in the graph of \( f_{k+1} \) is finite and contains a unique cycle.
These cycles are of length 2 a.s. and connect two exit points of order $k$, denoted \( x \) and \( y \), 
such that \( \mathrm{nn(k+1)}(x)=y \) and \( \mathrm{nn(k+1)}(y)=x \). See Figure \ref{Fi1} as an example of \( f_1 \)-graph.
\begin{figure}
\centering
	\includegraphics[width=6cm,height=5cm]{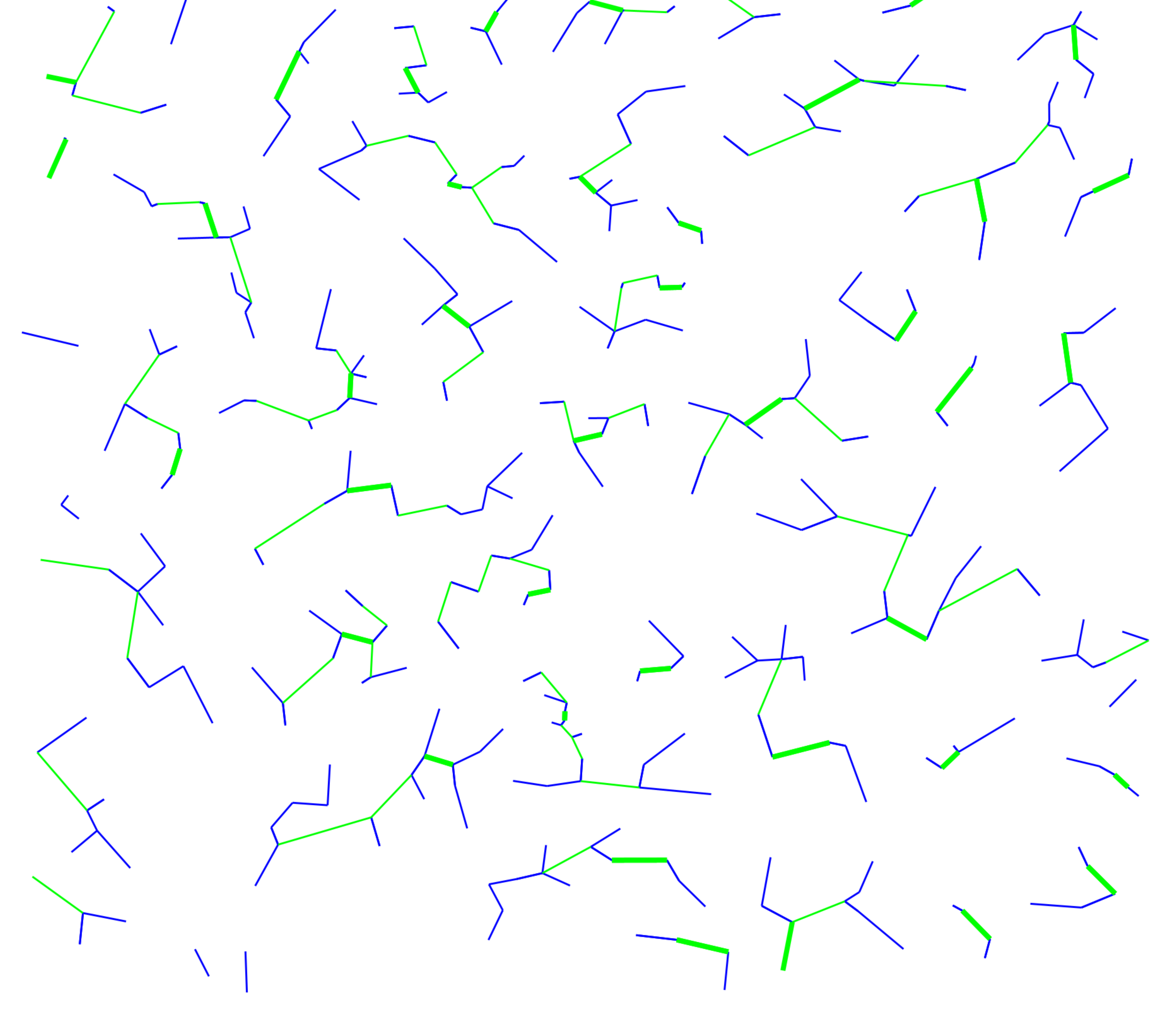} \hspace{1cm}
\includegraphics[width=6cm,height=5cm]{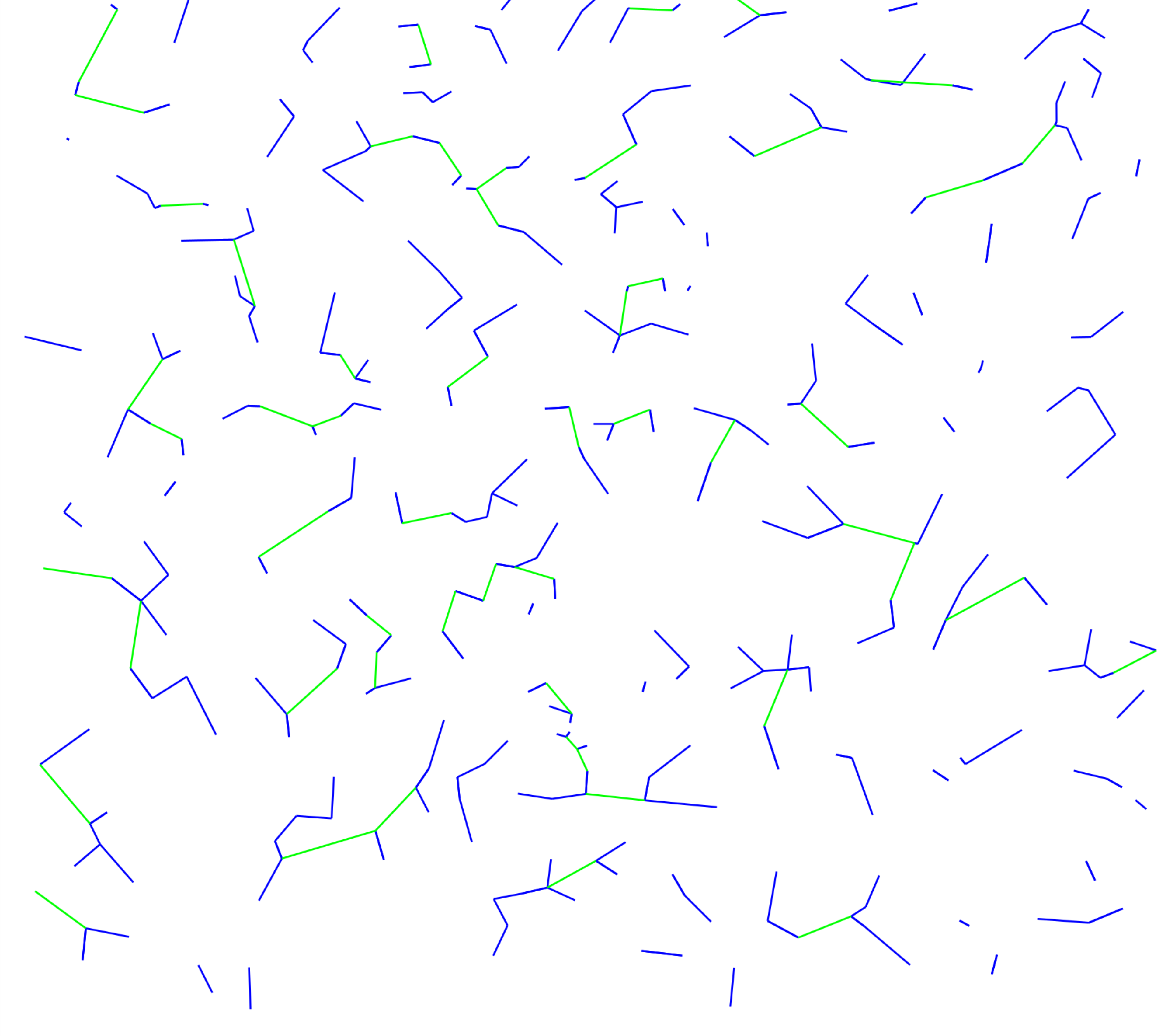}
\caption{\textbf{Left picture:} The \( f_1 \)-graph generated on the PPP with the \(1\)-cycles shown in green bold. 
\textbf{Right picture:} Cluster subtrees of order 1 obtained by deleting the edges of the \(1\)-cycles (clusters consisting of a single point are not visible here). 
}
\label{Fi1}
\end{figure}
\begin{remark}
\label{R013}
Denote by $\Phi_{e}^k$ the point process of exit points of order $k$, and let the intensity of $\Phi_{e}^0$ be $\rho/2$. 
Under the first level of the $\mathrm{CHN}^2$ clustering algorithm, each point in $\Phi_{e}^0$ is either a cluster head of 
order one or is connected to a cluster head of order one, and each resulting cluster contains exactly two cluster heads. 
Furthermore, when two cluster heads of order one belong to the same cluster of order one, one of them becomes an exit 
point of order one while the other does not. Therefore, the intensity of exit points of order one is less than 
$\rho / 4$. By the same reasoning, the intensity of exit points of order $k$ is less than $\rho / 2^{k+1}$.
Figure \ref{intensities2} shows simulation estimates of the intensity of exit points at level $k$. 
As seen in the figure, the intensity at each level is reduced by approximately a factor of three, 
suggesting that the estimated intensity of exit points at level $k$ is close to $\rho / 3^{k+1}$.
\end{remark}
\begin{figure}
\centering
\includegraphics[width=7cm,height=5cm]{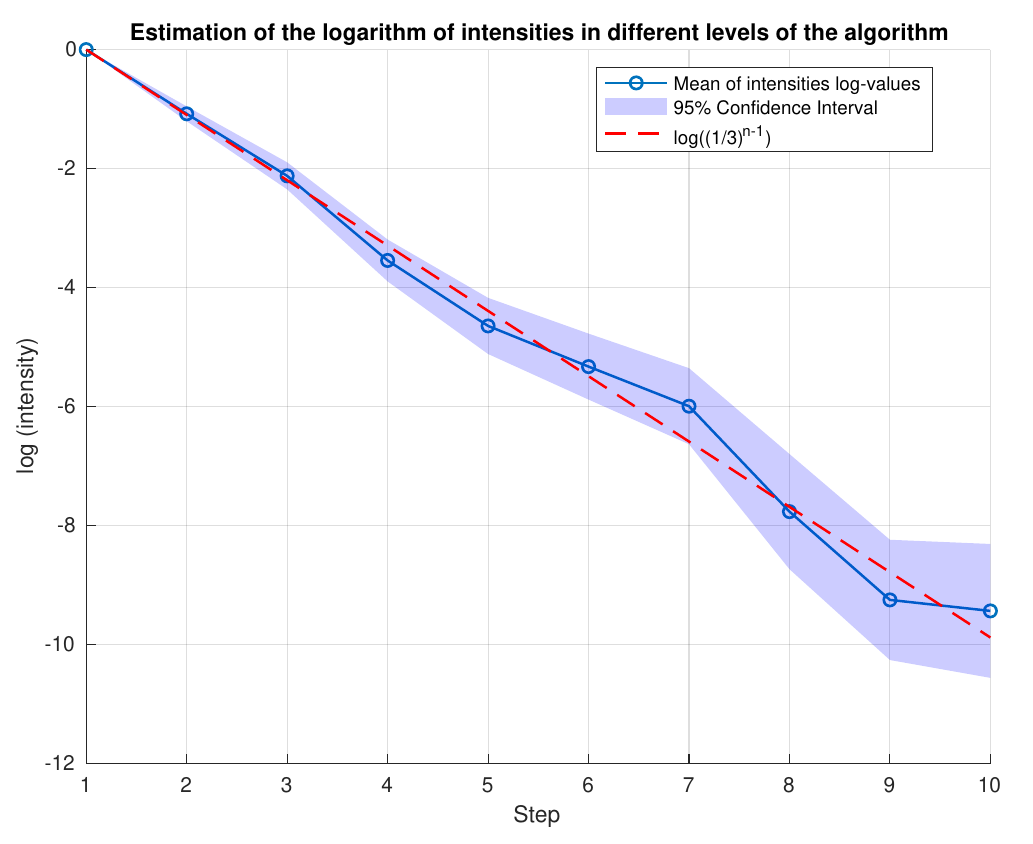} 
\caption{Estimation of the (logarithm of the) intensity of the exit points at different levels of the algorithm. The initial point process contains approximately 12,000 points. 
As shown in the plot, the intensity decays by a factor close to \(1/3\) at each level.
}
\label{intensities2}
\end{figure}
\subsection{Construction of the Limiting Point-Shift} 
\label{LIMITPS}
This subsection discusses the existence and properties of the limiting point-shift and point-shift graph
of the $\mathrm{CHN}^2$ clustering algorithm $\Phi^0$.
\begin{definition}
Consider the sequence $f_k$ of point-shifts constructed in Subsection~\ref{0011C} on $\Phi^0$. Since the sets $\{\Phi_a^{k+1}\setminus\Phi_a^k\}_{k\ge0}$, $\Phi_a^0=\varnothing$, form a partition of $\Phi^0$ (see Remark~\ref{partition}), one may define a limiting point-shift $f^\infty$ by assigning to each point the image it receives at the unique level at which it leaves the cluster-head process. The graph of this point-shift, referred to as the $f^\infty$-graph and denoted by $G$, is acyclic by construction.

The \(f^\infty\) point-shift is called the Centroid Hierarchical Nearest Neighbor Point-Shift ($\mathrm{CHN^2PS}$),
and its associated graph is called the Centroid Hierarchical Nearest Neighbor Eternal Family Forest
($\mathrm{CHN}^2$ EFF). When this forest is a tree, it is referred to as the Centroid Hierarchical Nearest
Neighbor Eternal Family Tree ($\mathrm{CHN}^2$ EFT).
\end{definition}
\begin{remark}
\label{partition}
Since $\Phi_c^{k+1}\subseteq \Phi_c^k$ for all $k$, the events
$\{0\in \Phi_c^k\}$ are decreasing under the Palm probability. Moreover, it is shown in Remark \ref{R013} that 
the intensity $\rho_k$ of $\Phi_c^k$ tends to zero as $k\to\infty$.
Hence,
\begin{equation}
\mathbb P^0\left(0\in\bigcap_{k\ge0}\Phi_c^k\right)
=
\lim_{k\to\infty}\mathbb P^0(0\in\Phi_c^k)
=
\lim_{k\to\infty}\rho_k
=
0.
\end{equation}
Therefore, almost surely every point of $\Phi^0$ eventually ceases to be a cluster head. Equivalently, the collection
$
\{\Phi_a^{k+1}\setminus\Phi_a^k\}_{k\ge0}
$
forms an almost-sure partition of $\Phi^0$.
\end{remark}


\begin{theorem}[Construction of the point-shift graph]
\label{Existence of the limit point-shift}
Under the Palm probability of $\Phi^0$ and rooted at zero,
the $\mathrm{CHN}^2$ EFF is the local weak limit of the $f_{n}$-graphs on $\Phi^0$ as $n\to \infty$. 
This limit graph is unimodular. 
\end{theorem}
\begin{proof}
This is assessed under the Palm probability of $\Phi^0$. By Remark~\ref{partition}, every point of $\Phi^0$ eventually leaves the cluster-head process. Consequently, the image of each point under the sequence of point-shifts $\{f_{n}\}_{n\in\mathbb{N}}$ changes only finitely many times. Fix $h>0$. Since the ball $B_h(0)$ contains only finitely many points of $\Phi^0$ a.s., there exists an a.s. finite random integer $N(h)$ such that no point of $B_h(0)$ belongs to $\Phi_c^{N(h)+1}$. Hence, no further modification of the graph occurs inside $B_h(0)$ after level $N(h)$, and the restriction of the graph $G^n$ to $B_h(0)$ remains unchanged for all $n\ge N(h)$. Therefore, for every $h>0$, the restriction of $G^n$ to $B_h(0)$ converges almost surely as $n\to\infty$. This defines a limiting rooted graph, denoted by $[G,0]$. For every $n$, the graph $[G^n,0]$ is unimodular under the Palm probability of $\Phi^0$ (see \cite{haji_mirsadeghi_baccelli2018}). Since unimodularity is preserved under local limits (see \cite{Baccelli2017}), the limiting graph $[G,0]$ is also unimodular.
\end{proof}
Before stating the next results, we briefly recall the terminology from unimodular network theory that will be used throughout the paper.

Let $F$ be a point-shift on a stationary point process $\phi$. Consider the equivalence relation
\[
x \sim y
\quad\Longleftrightarrow\quad
\exists\, m,n\ge0 \text{ such that } F^{m}(x)=F^{n}(y).
\]
The corresponding equivalence classes are called \emph{foils}. Intuitively, foils are the level sets transverse to the flow induced by the point-shift.

We shall use the Foil Classification Theorem and the No Infinite/Finite Inclusion Theorem from \cite{Baccelli2017}. These results classify the connected components of the $F$-graph according to the finiteness or infiniteness of both the component itself and its foils. In particular, the notation $\mathcal{F/F}$ refers to finite components with finite foils, whereas $\mathcal{I/I}$ refers to infinite components with infinite foils.
\begin{theorem}
\label{FFPROPERTY}
For each $n\in \mathbb{N}$, the connected components of the graph $[G^n,0]$ are all of type $\mathcal{F/F}$.
\end{theorem}
\begin{proof}
In Section \ref{ProofFF}, it will be shown that, for all $n$, there is no infinite path in $[G^n,0]$.
This fact demonstrates that, for each $n$, all the clusters of $[G^n,0]$ are finite. Note that the cycle
indicated in the properties of the $\mathcal{F/F}$ components (see Theorem 3.10 in \cite{Baccelli2017}) 
in each cluster of $f_n$-graph is the cycle of mutual nearest neighbor points, which is a cycle of length two a.s. 
\end{proof}
\begin{remark} 
Theorem \ref{FFPROPERTY} implies that clusters of level $n$ in the $\mathrm{CHN}^2$ 
algorithm are almost surely finite for each $n$.
\end{remark}
In the next proposition, it will be shown that all the connected components of the $\mathrm{CHN}^2$ EFF
share the same distribution.
\begin{proposition}
\label{Identicalcomponents}
Let $G$ be the $\mathrm{CHN}^2$ EFF under the Palm probability of $\Phi^{0} \subseteq \mathbb{R}^{d}$,
and let $f^\infty$ be $\mathrm{CHN^2PS}$ on $G$. Then the distribution of $[G, 0]$ is identical to 
the distribution of $[G, v]$, for all $v$ in the support of $\Phi^{0}$.
\end{proposition}
\begin{proof}
It is known that there exists a directed graph $G'$ with vertex set $\Phi^{0}$,
that is constructed in a deterministic and isometry-invariant way, which is a.s. a doubly infinite path 
(see \cite{Matching}). Consequently, a point-shift on $\Phi^{0}$ can be obtained by mapping each vertex 
to its subsequent vertex following the direction of this path. This point-shift is a bijection. 
By applying Mecke's Point Stationary Theorem (see \cite{haji_mirsadeghi_baccelli2018}), 
it can be concluded that the distribution of $[G, 0]$ is identical to the distribution of $[G,v]$.  
\end{proof}

\begin{corollary}
\label{IIconnected}
The connected components of the \(\mathrm{CHN}^2\) EFF, \([G,0]\), belong to the \(\mathcal{I/I}\) class
of the foil classification theorem.
In particular, under the PPP, each connected component is a one-ended tree and all of its foils (equivalence classes induced by the point-shift) are infinite a.s.
\end{corollary}
\begin{proof}
As mentioned previously, at each level \( k \), the equality \(\Phi^0 = \Phi^k_c \cup \Phi^k_a\) holds, where \(\Phi^k_c\) denotes the cluster heads and \(\Phi^k_a\) represents the acyclic points of order \(n\). For a point belonging to \(\Phi^k_a\) at some level \(k\), the number of its descendants is finite and does not change thereafter due to the construction of the $\mathrm{CHN}^2$ EFF. 

Since the intensity of \(\Phi^k_c\) tends to zero as \(k\) approaches infinity, the sets \(\{\Phi^{k+1}_a \setminus \Phi^k_a\}_{n \geq 0}\), with \(\Phi^0_a = \emptyset\), form a partition of \(\Phi^0\). Therefore, almost surely, there exists a level \(k\) such that zero belongs to \(\Phi^{k+1}_a \setminus \Phi^k_a\). Hence, its descendant tree is almost surely finite. Consequently, the connected component of zero is in the \(\mathcal{I/I}\) class. In other words, the cluster containing the origin in the $\mathrm{CHN}^2$ EFF graph is a one-ended tree. 

Proposition \ref{Identicalcomponents} implies that all other connected components are also in the \(\mathcal{I/I}\) class.
\end{proof}

\subsubsection{The Existence of a Last Universal Common Ancestor at infinity.}
Considering the genetic data of the species interpretation discussed in the introduction,
one can gain clearer insights by examining the properties of the $\mathrm{CHN}^2$ EFF graph.
Firstly, the $\mathcal{I/I}$ property of the graph shows that each species (point) belongs to an infinite phylogenetic tree.
The $\mathcal{I/I}$ property also shows that there exists a LUCA (Last Universal Common Ancestor) "at infinity",
from which all species in the cluster are descendants. This raises the question of whether there are multiple LUCAs
or just one in the clustering of the Poisson point process. This is equivalent to determining the number of 
connected components that exist in the $\mathrm{CHN}^2$ EFF.\\
Although the number of connected components is generally unknown, in dimension $d = 1$,
the following theorem shows that the $\mathrm{CHN}^2$ EFF has only one connected component.

\begin{theorem}[Connectivity of the limiting graph in dimension $d=1$]
\label{Dimention1}
In dimension $d=1$, under the Palm probability of $\Phi^{0} \subseteq \mathbb{R}$, the limiting $\mathrm{CHN}^2$ EFF is almost surely connected.
\end{theorem}

\begin{proof}
By Corollary~\ref{Existence of the limit point-shift}, the limiting $\mathrm{CHN}^2$ EFF $[G,0]$ is unimodular (as a local weak limit of the unimodular graphs $[G^n,0]$; see also~\cite{Baccelli2017}). 
By Corollary~\ref{IIconnected}, each connected component of $[G,0]$ is infinite almost surely.

Assume for the sake of contradiction that $[G,0]$ is not connected. 
Because the construction takes place on the line and each component, at every level, connects only to one of its neighbors on the left or on the right, each connected component occupies a consecutive block of vertices along the line (with no interlacing between components). 
Hence, distinct components occupy disjoint order-convex blocks along $\mathbb{R}$; since every component is infinite, there can be at most two such blocks—one extending to the left and one to the right.

Let $S$ be the set containing the two extremal vertices separating these components i.e., the rightmost vertex of the left component and the leftmost vertex of the right component. These two vertices are the only vertices that have degree 1 in the graph and can be defined in a covariant way. 

Hence, $S$ is a covariant subset of vertices in the sense of Definition~2.7 in~\cite{Baccelli2017}. Furthermore, the decomposition of $[G,0]$ into its connected components defines a covariant partition (Definition~2.9 in~\cite{Baccelli2017}). However, Lemma~2.11 (\emph{No Infinite/Finite Inclusion}) in~\cite{Baccelli2017} states that, in a unimodular network, a covariant subset cannot have a finite, non-empty intersection with an infinite element of a covariant partition. Here, $S$ intersects each infinite component of the partition in exactly one vertex, violating this property. Equivalently, by Corollary~2.12 in~\cite{Baccelli2017}, any covariant subset in an infinite unimodular network is a.s. either empty or infinite.Since $S$ contains exactly two vertices, this leads to a contradiction. Therefore, the limiting $\mathrm{CHN}^2$ EFF must be connected a.s.
\end{proof}

 The conjecture is that this graph is connected in each dimension $d$ when constructed from a PPP, i.e., there is only one LUCA in the $\mathrm{CHN}^2$ EFF phylogenetic tree of a PPP.

\subsubsection{The Cardinality of Descendant Trees}
The construction of the $\mathrm{CHN}^2$ EFF allows us to calculate bounds on the mean number 
of species within the descendant tree of a typical cluster at a specific level $k$.  Note that
the direction of the edges in the graph is from child to parent. 
The $\mathcal{I /I}$ property (of the connected component of zero), ensures a finite number of descendants 
for each vertex (species), so it enables us to quantify the number of species in the descendant tree. More precisely, let the intensity of the cluster heads of level $k$, $\Phi_{c}^{k}$, be $\rho_{k}$. 
Then the mean of the cardinality of the descendant tree of a point belongs to $\Phi_{c}^{k}$, is $1/\rho_{k}$. 
It is known from Remark \ref{R013} that  $\rho_{k}$ tend to zero, as $k$ goes to infinity.
\begin{proposition}
\label{Infinitedes}
Let $[G,0]$ be the $\mathrm{CHN}^2$ EFF under the Palm probability of $\Phi^{0} \subseteq \mathbb{R}^{d}$. 
Let $N^{0}$ denote the cardinality of the descendant tree of the root $0$. Then $\mathbb{E}[N^{0}]=\infty$. 
In particular, $N^{0}$ has a heavy tail.
\end{proposition}
\begin{proof}
Since $[G,0]$ is unimodular, the Mass Transport Principle applies (see~\cite{Aldous}).  
Consider the $\mathrm{CHN}^2$ point-shift $f$ oriented from each vertex to its parent.  
Because every component is one-ended, the parent pointers form an infinite ancestral ray almost surely.  
Applying the MTP to the transport that sends one unit of mass from each vertex to all its ancestors 
yields $\mathbb{E}[N^{0}]=\infty$.
\end{proof}
\begin{remark}
The result stated in Proposition \ref{Infinitedes} remains valid for all vertices in the $\mathrm{CHN}^2$ EFF. 
This is due to Proposition \ref{Identicalcomponents}, which ensures that all connected components share the same distribution.
\end{remark}
\section{$\mathrm{CHN}^2$ on Cox Point Processes}
\label{CoxSection}

The aim of this section is twofold. The first one is to show that the algorithms proposed here
can be applied to other point processes than Poisson.
The second one is to see how the algorithms in question can be used to detect data aggregates when present.

Clustering algorithms are expected to detect aggregates of points in a dataset. 
Working with a Poisson point process is useful in this context because the PPP can be viewed as a null hypothesis,
i.e., a baseline model that has no structural underlying aggregates and where fine random local
properties are leveraged to nevertheless recursively cluster data in function of the dataset only.
To examine the behavior of the \(\mathrm{CHN}^2\) algorithm in scenarios where structural aggregates do occur, 
a natural testbed is that of Cox point processes. 

Specifically, Example~\ref{Ex.Cox} is constructed so that each realization contains distinct aggregates of points (see Remark~\ref{aggragates}) and serves to illustrate how the algorithm behaves when such aggregates are present. The properties established in Section~\ref{Ch:apter7}, including the finiteness of clusters at every hierarchical level and the existence of a unimodular limiting graph, continue to hold in this setting (Theorem~\ref{thm:transfer-no2dc}). In addition, the simulations suggest that the presence of aggregates can be identified through abrupt changes in the inter-cluster distances relative to a Poisson baseline; see Subsection~\ref{simulation2}.

\begin{remark}
\label{aggragates}
In many clustering algorithms, the term \emph{cluster} refers to a distinct subgroup of data points that
are more similar to one another than to points outside that subgroup. In contrast, in this paper, 
the word \emph{cluster} simply denotes a connected component at a given hierarchical level. 
To avoid confusion between these two usages, the term \emph{aggregates of points} will be used in this 
paper to denote data with distinct groups of similar objects.
\end{remark}

\begin{example}[Cox point process]
\label{Ex.Cox}
A point process \(N\) on a state space \(\mathcal{S}\) is called a \emph{Cox point process} 
(driven by the random measure \(M\)) if, conditioned on \(M\), it is a Poisson point process with intensity measure \(M\) (see, e.g.,~\cite[Section~2.3]{Bartek}). 
Equivalently, \(N\) is a Poisson point process whose intensity measure is itself random.

Let $\Xi \subseteq \mathbb{R}^d$ be a random closed subset such that
$
\mu(\Xi\cap B)>0
$
with positive probability for every bounded Borel set $B\in\mathcal B(\mathbb R^d)$, where
\(\mathcal{B}(\mathbb R^d)\) denotes the Borel \(\sigma\)-algebra of \(\mathbb R^d\). Define the random measure $M$ on $\mathbb R^d$ by
\begin{equation*}
M(B)=\mu(B\cap \Xi),
\qquad B\in\mathcal B(\mathbb R^d),
\end{equation*}
where \(\mu\) denotes the Lebesgue measure on \(\mathbb{R}^d\).  
Let \(N\) be the Cox point process driven by \(M\).  
Conditioned on the realization of \(\Xi\), \(N\) is a homogeneous Poisson point process restricted to the random region \(\Xi\); in particular, it does not admit any second-order descending chain.  
If the random set \(\Xi\) is stationary in law (for example, generated from a stationary marked point process), then \(N\) is also stationary.

\medskip
\noindent\emph{Special case.}  
In the simulations of Figure ~\ref{fig:cox-vertical}, the random set \(\Xi\) is chosen as a union of random balls,  
\(\Xi = \bigcup_{x \in \xi} B_{r_x}(x)\), where \(\xi\) is a stationary point process and \(\{r_x\}_{x \in \xi}\) are i.i.d.\ random radii.  
This construction illustrates the behavior of the $\mathrm{CHN}^2$ clustering algorithm when applied to Cox point processes exhibiting spatial aggregation.
\end{example}

The following theorem shows that all properties established for the homogeneous PPP extend to this Cox setting.
\begin{theorem}
\label{thm:transfer-no2dc}
Let $N$ be a stationary point process in $\mathbb{R}^d$ that almost surely admits no second-order
descending chain (Def.~\ref{017}), and assume the same general-position (no ties) conditions
as in Section~\ref{Ch:apter7}. This includes, in particular, any stationary sub--point process of a homogeneous PPP and the Cox point process defined in Example \ref{Ex.Cox}. Then the following properties hold for $\mathrm{CHN}^2$ on $N$:
\begin{enumerate}[label=(\roman*)]
\item For each $n\in\mathbb{N}$, the connected components of the $f_n$-graph
are almost surely in the $\mathcal{F/F}$ class (finite with a unique cycle). \label{it:FF}
\item The sequence of $f_n$-graphs converges locally weakly to a unimodular limit graph. \label{it:limit}
\item Under the Palm version of $N$, the connected component of the root
in the limiting graph belongs to the $\mathcal{I/I}$ class (almost surely one-ended with all foils infinite). \label{it:II-root}
\end{enumerate}
Moreover, if there exists a nontrivial bijective point-shift on $N$
(so that $[G,o]$ and $[G,v]$ have the same distribution for every vertex $v$),
then property~\ref{it:II-root} holds for \emph{all} components;
that is, every connected component of the limit is $\mathcal{I/I}$ almost surely.
\label{it:II-all}
\end{theorem}

\begin{proof}
\textbf{(i)~Finiteness of the pre-limit graphs.}
The proof follows by repeating the argument of Theorem~\ref{FFPROPERTY} with $\Phi^0$ replaced by $N$. That proof only uses (a) stationarity and (b) the absence of (second-order) descending chains to rule out infinite paths (via the reduction to second-order chains as in Proposition~\ref{p005}). Both hypotheses hold here by assumption, hence every $f_n$-component is a.s.\ $\mathcal{F/F}$.

\medskip
\textbf{(ii)~Unimodular local weak limit.}
The proof follows by the proof of Corollary~\ref{Existence of the limit point-shift}: the vanishing intensity of cyclic points across levels implies local stabilization of the $f_n$-graphs on bounded windows, yielding a local weak limit $[G,0]$; unimodularity of the limit then follows from unimodularity of the prelimits and closedness of unimodularity under local weak limits (see~\cite{Baccelli2017}).

\medskip
\textbf{(iii)~One-endedness of the component of the root.}
This follows from adapting the proof of Corollary~\ref{IIconnected}:
under the Palm version of $N$, evaporation holds because the intensity of exit points
tends to zero as the level increases.
Hence, the descendant tree of the root is finite while the ancestral ray is infinite.
By the foil-classification theorem for unimodular networks
(\cite[Thm.~3.10]{Baccelli2017}), the rooted component is $\mathcal{I/I}$ a.s.

\medskip
\noindent
    \textbf{Upgrading to all components under a bijective point-shift.}
When a bijective point-shift exists, the distributions of $[G,o]$ and $[G,v]$
are identical for every vertex $v$ (cf.\ Proposition~\ref{Identicalcomponents} for the PPP case).
The $\mathcal{I/I}$ property then extends from the Palm root to every vertex;
so all components are $\mathcal{I/I}$ almost surely.
In particular, this condition is satisfied by the stationary Cox point process
of Example~\ref{Ex.Cox}, since the existence of a bijective point-shift for
any stationary non-equidistant point process follows from a theorem by
Timár \cite[Theorem~1]{timar2009}.
\end{proof}

\begin{figure} 
  \centering
    \includegraphics[height=.282\textheight,keepaspectratio]{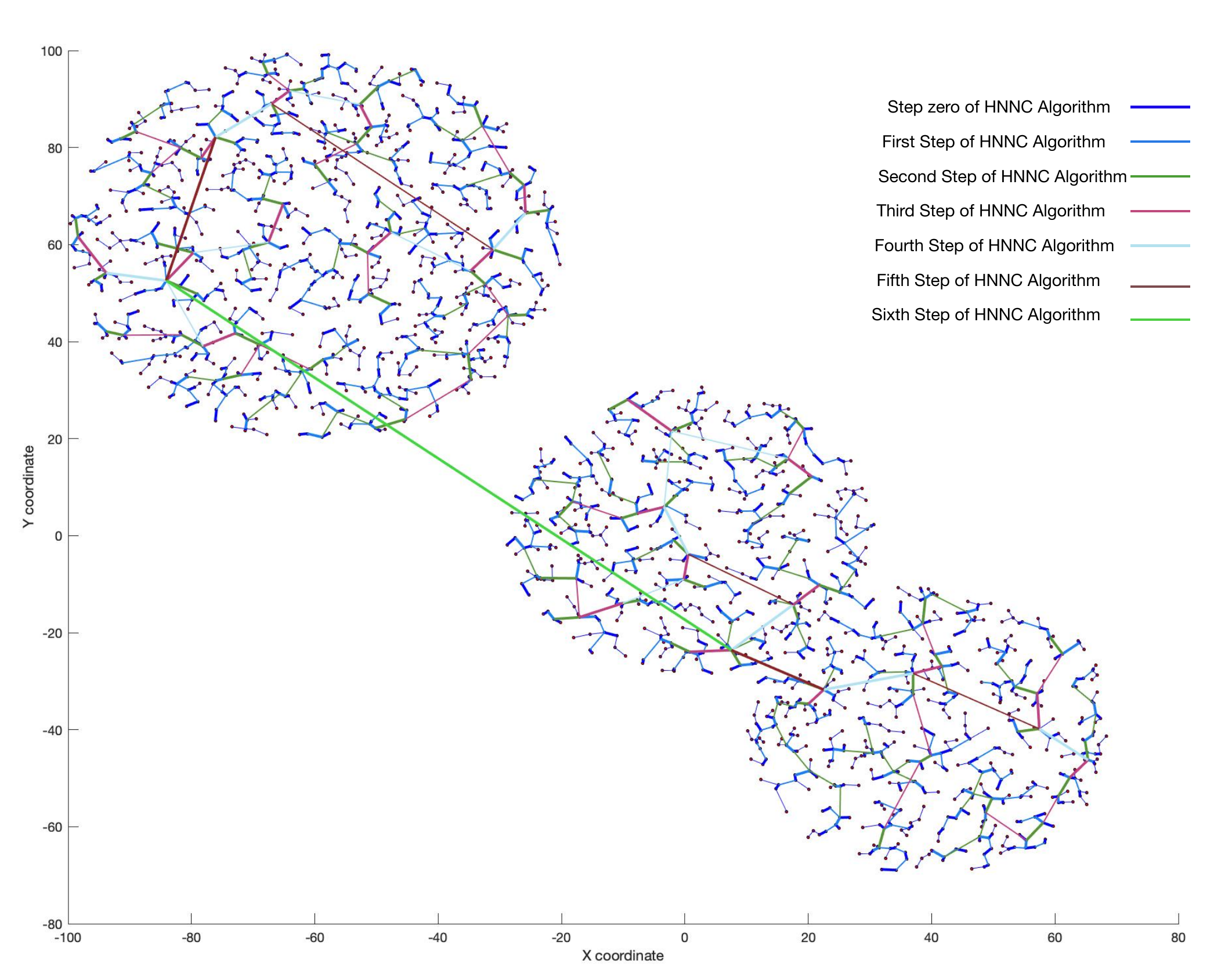}
    \includegraphics[height=.282\textheight,keepaspectratio]{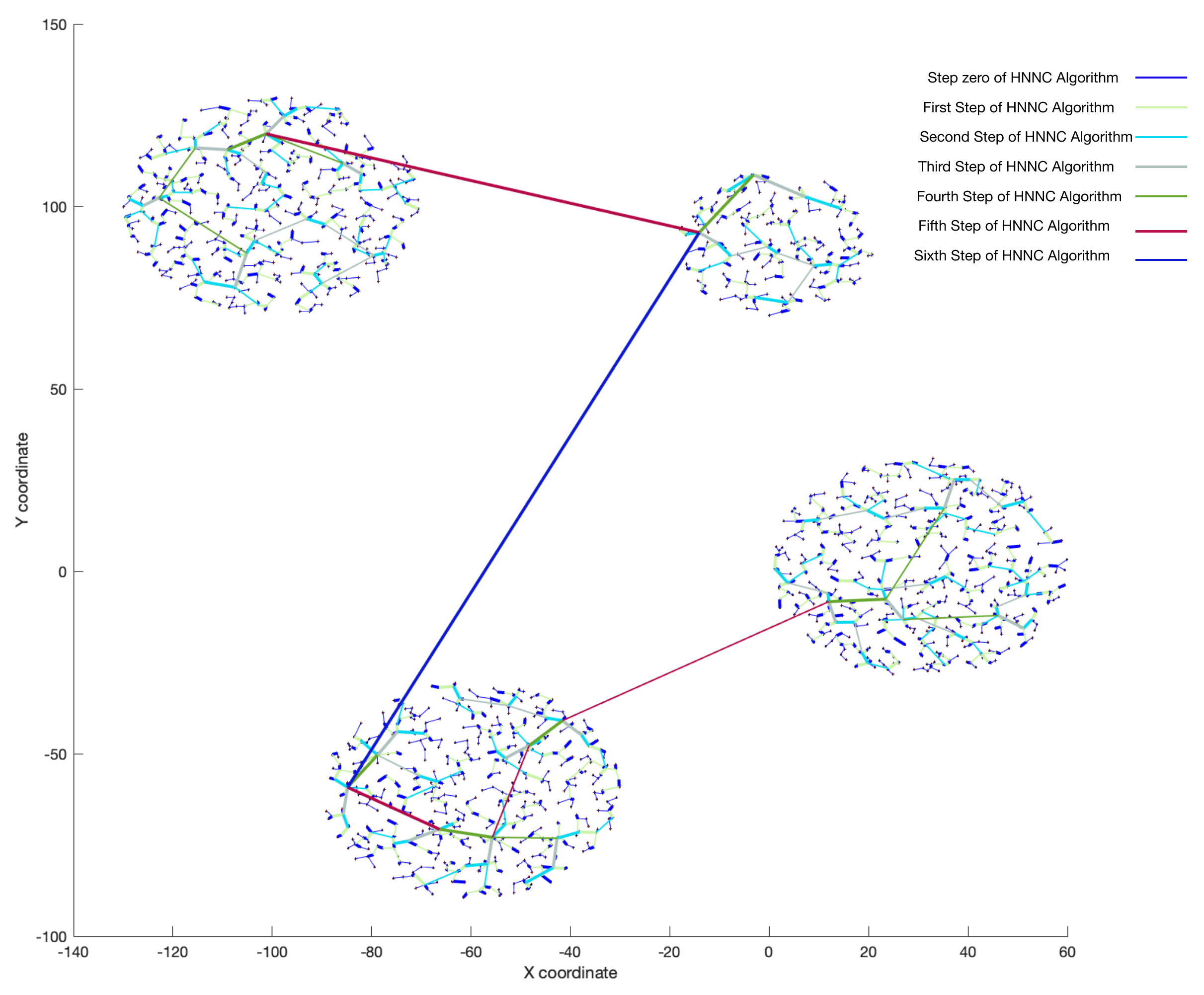}
    \caption{ CHN\(^2\) on a Cox PP realization.
    Left: A realization of the Cox PP (Example~\ref{Ex.Cox}) with \(S=\{3\) points\(\}\) and \(r=\{40,20,30\}\); \(\sim\)2000 points; merges after six steps.
    Right:
    A realization of the Cox PP (Example~\ref{Ex.Cox}) with \(S=\{4\) points\(\}\) and \(r=\{30,20,30,30\}\); \(\sim\)2500 points; merges after seven steps.}
  \label{fig:cox-vertical}
\end{figure}

\subsection{\texorpdfstring{Stopping Criteria for Detecting Data Aggregation in $\mathrm{CHN}^2$ Clustering}{Stopping Criteria for Detecting Data Aggregation in CHN2 Clustering}}
\label{simulation2}
A classical question in hierarchical clustering is that of identifying aggregates through the clustering process. It is proposed here to use stopping rules based on the lengths of the edges introduced by the clustering procedure. These stopping rules are defined relative to a Poisson dataset, which serves as a null model without intrinsic aggregation.

To this end, the mean distances between clusters at successive levels of the $\mathrm{CHN}^2$ hierarchy are computed and plotted in Figure~\ref{coxmeandistance}. The mean inter-cluster distances exhibit a pronounced increase at a specific level of the hierarchy. For the dataset shown on the left of Figure~\ref{fig:cox-vertical}, this increase occurs around level (6), whereas for the dataset shown on the right it occurs around level (5). These jumps suggest natural stopping levels for the clustering procedure and provide evidence of spatial aggregation relative to the Poisson baseline.

\begin{figure}
\centering
	\includegraphics[width=6.6cm,height=5cm]{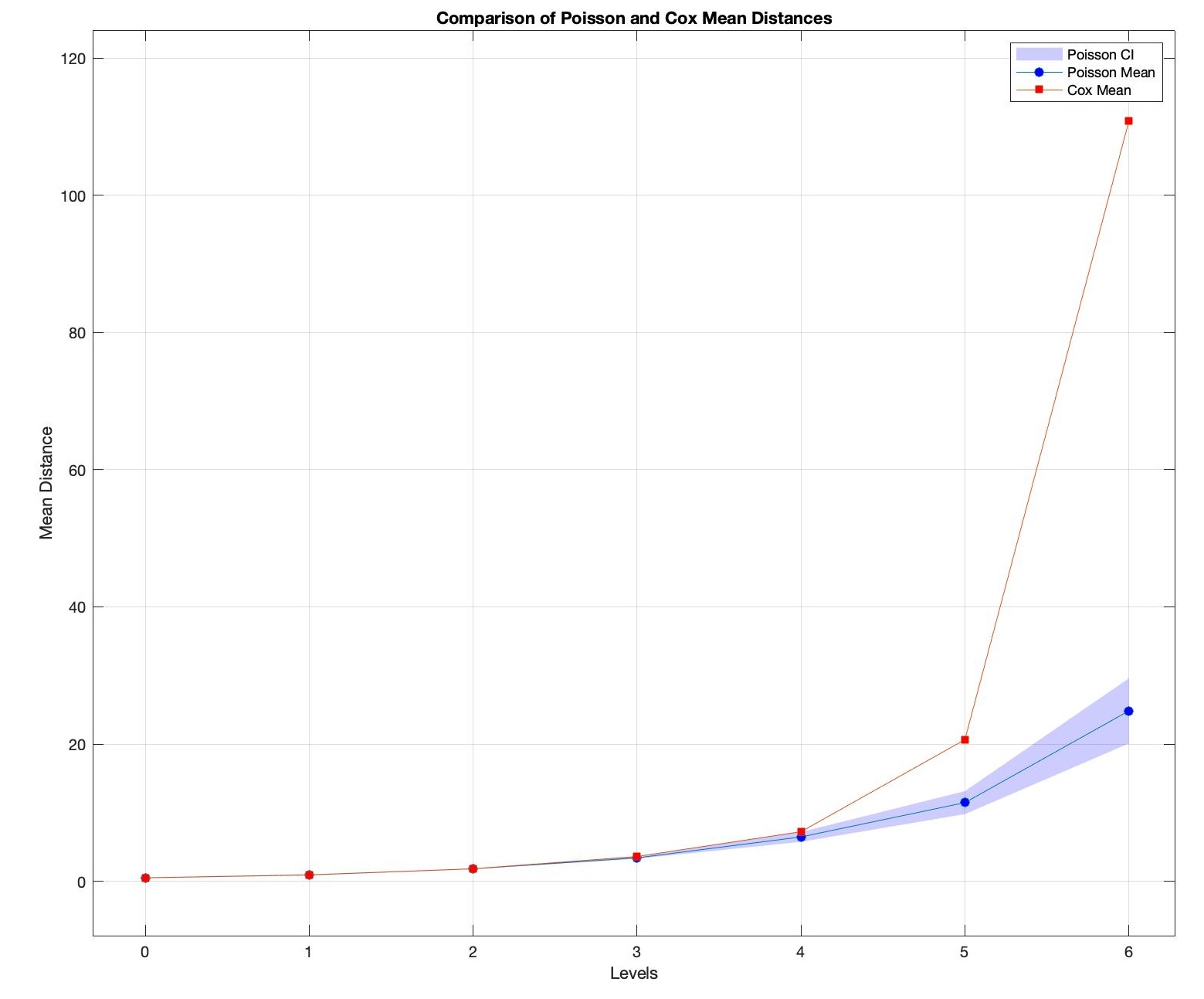} \hspace{0.1cm}
\includegraphics[width=6.6cm,height=5cm]{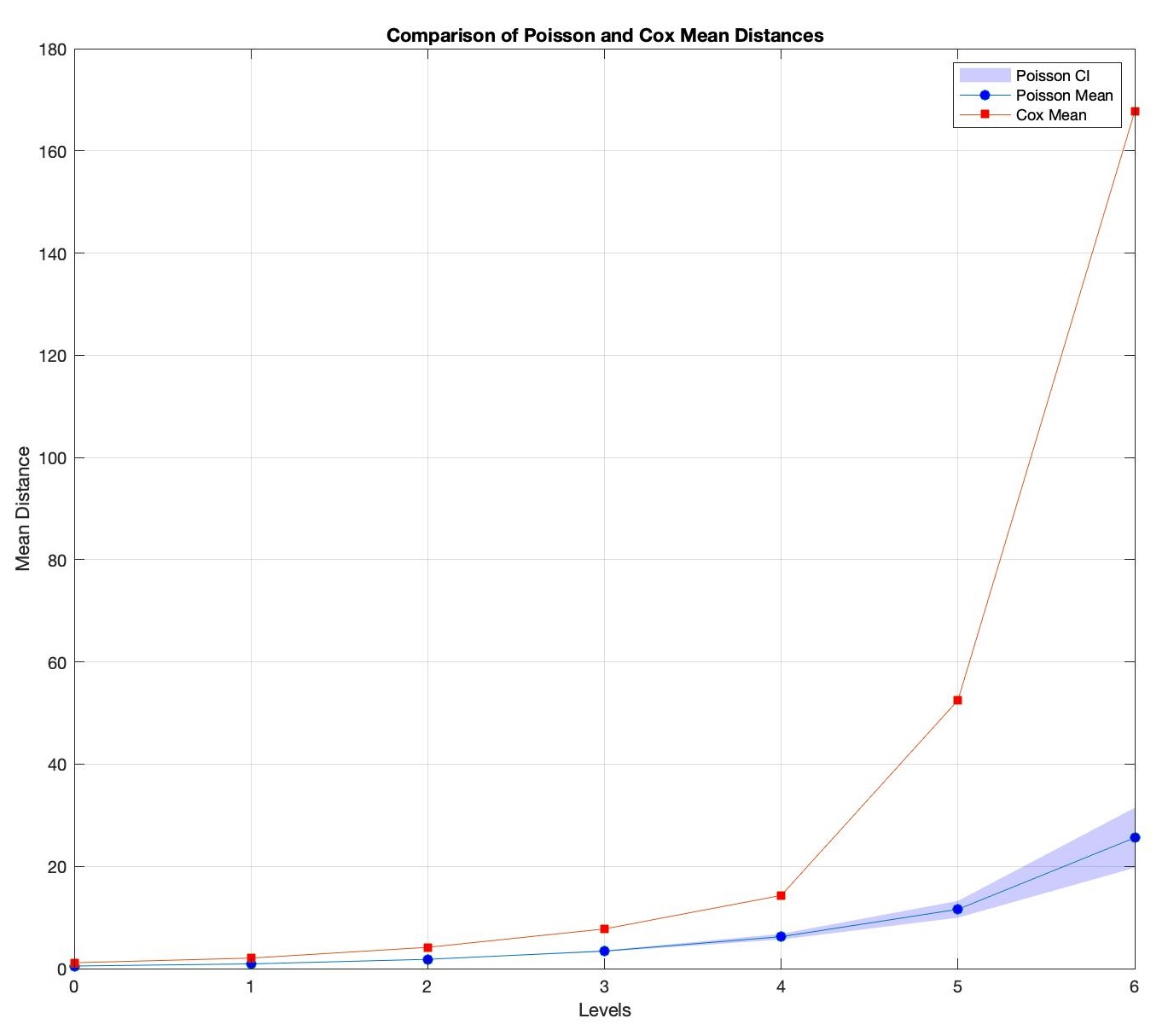}
\caption{
Mean inter-cluster distances at successive levels of the $\mathrm{CHN}^2$ hierarchy for the Cox datasets of Figure~\ref{fig:cox-vertical} and for a Poisson baseline. Left: a pronounced increase occurs around level $6$. Right: the first significant increase occurs around level $5$. These jumps suggest natural stopping levels for the hierarchical clustering procedure.
}
\label{coxmeandistance}
\end{figure}
\section{\texorpdfstring{Proof of $\mathcal{F/F}$ property of the pre-limits graphs of $\mathrm{CHN}^2$}{Proof of F/F property of the pre-limits point-shift graphs of CHN2}}
In Section~\ref{0011C} and Theorem~\ref{FFPROPERTY}, it was claimed that for every $k\in\mathbb N$, the graph associated with the point-shift $f_k$ contains only finite connected components. The purpose of this section is to prove this result.

We first establish the claim for the nearest-neighbor point-shift $f_0$. We then introduce the notion of a second-order descending chain and show that any infinite path in the $f_1$-graph would induce such a chain. Then it is proved that second-order descending chains do not exist in the homogeneous Poisson point process, it follows that all components of the $f_1$-graph are finite. Finally, the same argument is extended inductively to all $f_k$-graphs.

\label{ProofFF}
\begin{proposition}
\label{F/Fproperty}
The graph of the point-shift \(f_0\) (the \(f_0\)-graph) is unimodular, and all its connected components
belong to the \(\mathcal{F/F}\) class in the sense of \cite{Baccelli2017}; that is, each component is
almost surely finite and contains exactly one cycle. Equivalently, both the component and its foils
(equivalence classes induced by the point-shift) are a.s.\ finite.
\end{proposition}
\begin{proof}
The unimodularity of the \( f_0 \)-graph follows from the fact that it is obtained 
as a covariant shift on a stationary point process under the Palm probability 
(see \cite{Baccelli2017, Aldous}).
To see that all the connected components of this graph is finite, note that any infinite
path in the \( f_0 \)-graph on the PPP 
would form a descending chain, i.e., an infinite sequence of points whose consecutive 
distances form a decreasing sequence 
(see \cite{Meester, Matching, Daley}). 
However, it is shown in \cite{Daley} that the homogeneous PPP does not admit any 
descending chain almost surely. 
Therefore, the connected components of the \( f_0 \)-graph are almost surely finite 
and consist of directed trees attached to a single cycle connecting two mutual 
nearest neighbors of \( \Phi^0 \).
\end{proof}
For simplicity in the case $k>0$, it is first shown that the result holds for $k = 1$. Then, it will be shown that the same result holds for all $k$.
Before this, \textit{Second-order descending chains} are introduced. It is shown in Proposition \ref{004SOD} that such chains do not exist in the Poisson point process and  
in Proposition \ref{p005} that an infinite path in the $f_{1}$-graph is a second-order descending chain.
\begin{definition}
\label{017}
Let $N$ be a point process. Let $\{x_{n}\}_{n \in \mathbb{N}}$ be an infinite sequence of points in $N$ without repetition. 
This sequence is called a \textit{second-order descending chain} if, for all $i\geq 2$, the inequality
$d_{i} < \max (d_{i-1},d_{i-2})$ holds. 
\end{definition}
\begin{remark}
A \emph{descending chain} $\{x_{n}\}_{n \in \mathbb{N}}$ in a point process $N$ is defined by the property 
that for all $i \geq 1$, the inequality $d_i < d_{i-1}$ holds, where $d_i = |x_{i+1} - x_i|$ 
(see \cite{Meester} and \cite{Daley}). Any descending chain in $N$ is a second-order descending chain; 
however, the converse does not hold.
\end{remark}

\begin{proposition} 
\label{004SOD}
A $d$-dimensional Poisson point process, $\Phi$, contains no second-order descending chain a.s.
\end{proposition}
\begin{proof}
By classical arguments, it is enough to show that if $\Phi$ is a homogeneous PPP,
the probability of having a second-order 
descending chain in $\Phi +\delta_0$, starting from $0$, is zero. 
Consider the following events:
\begin{equation*}
A=\{ \text{ there exists a second-order chain, $\{x_{n}\}_{n \in \mathbb{N}}$, in $\Phi +\delta_{0}$, with $x_{0}=0$} \},
\end{equation*}

\begin{equation*}
A_R = \left\{
\begin{array}{l}
\text{there exists a second-order chain } \{x_n\}_{n \in \mathbb{N}} \text{ in } \Phi + \delta_0, \text{ with } x_0 = 0, \\
\text{such that } d_0, d_1 < R
\end{array}
\right\}.
\end{equation*}
where $R$ is an integer.
Since $A=\cup _{R=1}^{\infty} A_{R}$, it is enough to show that $\mathbb{P}(A_{R})=0$ for all $R$. For $n\ge 3$,
let $A_{R,n}$ be the event of existence of a second-order descending chain, $\{x_i\}_{i=0}^{n}$  
of length (number of edges) $n$, that is, a chain in $\Phi +\delta_{0}$, with $x_{0}=0$, satisfying the property 
$d_{i} < \max (d_{i-1},d_{i-2})$ for all $2\le i\le n-1$, and such that $d_{0},d_{1}<R$. 
Then one can write $A_{R}=\cap_{n\geq 3} A_{R,n}$. Since for all $i$, $A_{R,i}\supseteq A_{R,i+1} $, so $\mathbb{P} (A_{R})= \lim _{n \to \infty} \mathbb{P} ( A_{R,n}).$
Thus it is enough to show that $\lim _{n \to \infty} \mathbb{P} ( A_{R,n})=0$.

For $n\ge 3$, let $X_{R,n}$ denote the number of second-order descending chains with length $n$
in $\Phi+\delta_{0}$ with $x_{0}=0$ and with $d_{0},d_{1}<R$.  
Since $\mathbb{P}(A_{R,n}) \leq \mathbb{E}[X_{R,n}]$, it is enough to show that 
$\lim_{n \to \infty} \mathbb{E}[X_{R,n}]=0$. 

For $n\ge 1$, consider the function $g_{n+2}(x,\Phi)$ equal to the number of second-order descending chains with 
length $n+2$, $x_{0}=0$, and $x_{1}=x\ne 0$, under the assumption that 0 and $x$ are in the support of $\Phi$. Then 
\begin{align*}
	\mathbb{E}[X_{R,n+2}] & = \mathbb{E}  	[\, \sum _{x \in \mbox{supp}(\Phi) \cap B_{R}(0)} g_{n+2}\left(x,\Phi+\delta_{0}\right) ]\, \\
	& =  \lambda \int _{B_{R}(0)} \mathbb{E} [\, g_{n+2}\left(x, \Phi +\delta_{0}+\delta_x\right) ]\, {\mathrm d}x , 
\end{align*} 
where the last equality follows from Campbell's formula and Slivnyak's theorem. 
For all $\Phi$ with a support containing $\{0\}$ and $\{x\}$, 
consider the function $g^{\prime}_{n+2}(y,\Phi)$ equal to 
the number of second-order descending chains with length $n+2$, without repeated vertices
and such that $x_{0}=0$, $x_{1}=x\ne 0$, and $x_{2}=y$, with $y\ne 0$ and $y\ne x$,
when assuming that $y$ is in the support of $\Phi$. By the same arguments as above, 
\begin{align*}
\mathbb{E}[X_{R,n+2}] & =  \lambda \int _{B_{R}(0)} \mathbb{E} [\, \sum _{y \in B_{R}(x) 
	\cap\mbox{supp}(\Phi), y\ne 0, y\ne x} g^{\prime}_{n+2}\left(y,\Phi+\delta_0+\delta_x\right) ]\, {\mathrm d}x \\
	&= \lambda \int _{B_{R}(0)} \lambda \int _{B_{R}(x)}  \mathbb{E} [\, g^{\prime}_{n+2} \left(y,\Phi +\delta_0+\delta_x+\delta_y\right) ]\, {\mathrm d}y {\mathrm d}x. 
\end{align*} 
The definition of second-order descending chains, together with the fact that we consider chains 
without repeated vertices and the stationarity of $\Phi$, imply that the last equation is equal to 
 \begin{equation*}
	 \lambda^{2} \int_{B_{R}(0)}   \int_{B_{R}(x)}\mathbb{E}[\, X_{\max(\lVert x \rVert,\lVert y-x \rVert),n}]\, {\mathrm d}y {\mathrm d}x. 
 \end{equation*}
So the following recurrence equation holds for the events $X_{R,n}$: 
\begin{equation}
\label{Recursive}
\mathbb{E}[X_{R,n+2}]= \lambda^{2} \int_{B_{R}(0)}  
	\int_{B_{R}(x)}\mathbb{E}[\, X_{\max(\lVert x \rVert,\lVert y-x \rVert),n}]\, {\mathrm d}y {\mathrm d}x. 
\end{equation}
Using this recursive equation, it will be shown in Lemma \ref{p002} that equation \eqref{Number} holds for $\mathbb{E}[X_{R,n}]$, so when $n$ goes to infinity, the number of descending chain of length $n$ in the Poisson point process goes to zero.
\end{proof}
\begin{lemma}
\label{p100}
Let $\Phi$ be a $d$-dimensional Poisson point process of intensity $\lambda$, and
	let $X_{R,n}$ be the number of second-order descending chains of length $n$ with $\max(d_0,d_1)<R$.
Then for all $n \in \mathbb{N}$,
\begin{equation}
\label{Number}
\mathbb{E}[X_{R,n}]
=
\frac{\bigl(\lambda^{2}\,\omega_{d}^{2}\,R^{2d}\bigr)^{\lfloor n/2\rfloor}}
     {\lfloor n/2\rfloor!}
\;\times\;
\bigl(\lambda\,\omega_{d}\,R^{d}\bigr)^{\,n \bmod 2},
\end{equation}
where $\omega_{d}$ is the volume of the unit $d$-dimensional ball.
\end{lemma}

\begin{proof}
The proof is by induction on $n$.
Since the recurrence relation \eqref{Recursive} increases $n$ by $2$ at each step,
there are two separate induction "tracks": one starting from $n=0,2,4,\dots$ and the other starting from $n=1,3, 5, \dots$. 

For even numbers the base case is $n=0$. $\mathbb{E}[X_{R,0}]$ is the number of second-order descending chains 
with length $0$ in $\Phi \cup {0}$, with $x_{0}=0$, which is equal to 1, so the base case holds for the even numbers. 
Assume that \eqref{Number} holds for an even number $n$. It then follows from \eqref{Recursive} that
\begin{align}
\label{p002}
	\mathbb{E}[X_{R,n+2}]&= \lambda^{2} \int_{B_{R}(0)}   \int_{B_{R}(x)}\mathbb{E}[\, X_{max(\lVert x \rVert,\lVert y-x \rVert),n}]\, {\mathrm d}y {\mathrm d}x \nonumber \\
	&=\frac{\lambda^{n+2}\omega_{d}^{n}}{n/2 !}  \int_{B_{R}(0)}  \int_{B_{R}(x)} (\max(\lVert x \rVert,\lVert y-x \rVert))^ {nd}{\mathrm d}y {\mathrm d}x .
\end{align}
The inner integral is equal to 
\begin{align}
\label{p003}
	\int_{B_{R}(x)} (\max(\lVert x \rVert,\lVert y-x \rVert))^ {nd} {\mathrm d}y 
	&= \int_{\theta \in {\mathbb S}^{d-1}} \int_{0}^{R} (\max(\lVert x \rVert, r))^ {nd} r ^{d-1} dr d\theta \nonumber  \\
&= \int_{ \theta \in {\mathbb S}^{d-1}} \left( \, \int_{0}^{\lVert x \rVert} \lVert x \rVert^ {nd} r ^{d-1} dr + \int_{\lVert x \rVert}^{R} r ^{(n+1)d-1} dr \right)\, d\theta \nonumber   \\
&= \frac{\omega_{d}}{n+1} (\, R^{(n+1)d}+ n\lVert x \rVert^{(n+1)d} ), 
\end{align}
using the change of variables \(y-x=r\theta\), where \(r\ge0\),  \(\theta\in S^{d-1}\) and ${\mathbb S}^{d-1}$ denotes the $d$-sphere.
Equation \eqref{p002} together with  \eqref{p003} give
\begin{align*}
	\mathbb{E}[X_{R,n+2}]&= \frac{\lambda^{n+2}\omega_{d}^{n+1}}{(n+1)(n/2 !)}  \int_{B_{R}(0)}   (\, R^{(n+1)d}+ n\lVert x \rVert^{(n+1)d} ) {\mathrm d}x 
=  \frac{\lambda^{n+2}\omega_{d}^{n+2}}{\frac{n+2}{2} !} R^{(n+2)d}, 
\end{align*}
which completes the proof for even numbers.

\textcolor{black}{For odd \(n\), the base case is \(n=1\). Observe that
\(\mathbb{E}[X_{R,1}]\), the expected number of second-order descending chains of length \(1\),
is \(\lambda\,\omega_d\,R^d\). This matches \eqref{Number} because \(\lfloor 1/2\rfloor = 0\)
and \((1 \bmod 2) = 1\). Then, by applying the same recurrence relation \eqref{Recursive},
which again advances indices in steps of two, one proves the result for all odd \(n=1,3,5,\dots\).
Hence the inductive argument for odd \(n\) is identical to the even case,
simply starting from \(n=1\) instead of \(n=0\).}
\end{proof}
\begin{definition} The definition of $f_{1}$ implies that the $f_{1}$-graph, as a directed graph, 
contains two distinct types of edges on $\Phi_{c}^{0}$. The first is that of edges between the nearest 
point w.r.t. $\mathfrak d$, referred to as $\mathfrak d$-type edges of level 0. 
The other type is that of edges that connect an exit point, $x$, of level $0$ to ${\mathrm {NN0}}(x)$. 
These edges are referred to as $\delta$-type edges of level $1$. That is, edges inherited from the nearest-neighbor graph are  $\mathfrak d$-type edges, whereas edges connecting cluster representatives are $\delta$-type edges.
\end{definition}
\begin{proposition}
\label{p005}
Any infinite path in the $f_{1}$-graph is a second-order descending chain. 
\end{proposition}
 \begin{proof}
Assume that the $f_{1}$-graph has an infinite component. 
Since the $f_0$-graph has all its components finite (Proposition \ref{F/Fproperty}),
there must be an infinite path $\mathcal{P}$ on \(\{S^0_i\}_i\) (see Definition \ref{Exit0}).
Let $e_{i}$ be the $i$-th edge in $\mathcal{P}$. Then $e_{i}$ can be either of $\mathfrak d$ or $\delta$ type. 
If $e_{i}$ is of the $\mathfrak d$-type, then, there are two possibilities for $e_{i-1}$ and $e_{i-2}$:
\begin{enumerate}
\item $e_{i-1}$ and $e_{i-2}$ are both $\delta$-type edges (see Figure \ref{F5}.a).  
In this case, since $e_{i}$ is a MNN0 edge,
then $\|e_{i}\|<\|e_{i-1}\|$. Thus $\|e_{i}\|< \max (\|e_{i-1}\|, \|e_{i-2}\|)$. 
\item $e_{i-1}$ is a $\delta $-type edge and $e_{i-2}$ is a $\mathfrak d$-type edge 
(see Figure \ref{F5}.b). Since $e_{i-2}$ and $e_{i}$ are both edges between MNN0 points, this gives $\|e_{i-1} \|> \|e_{i}$ which implies that $ \|e_{i}\|< \|e_{i-1} \|= \max (\|e_{i-1}\|, \|e_{i-2}\|)$.
\begin{figure}
\centering
\includegraphics[width=.7\textwidth]{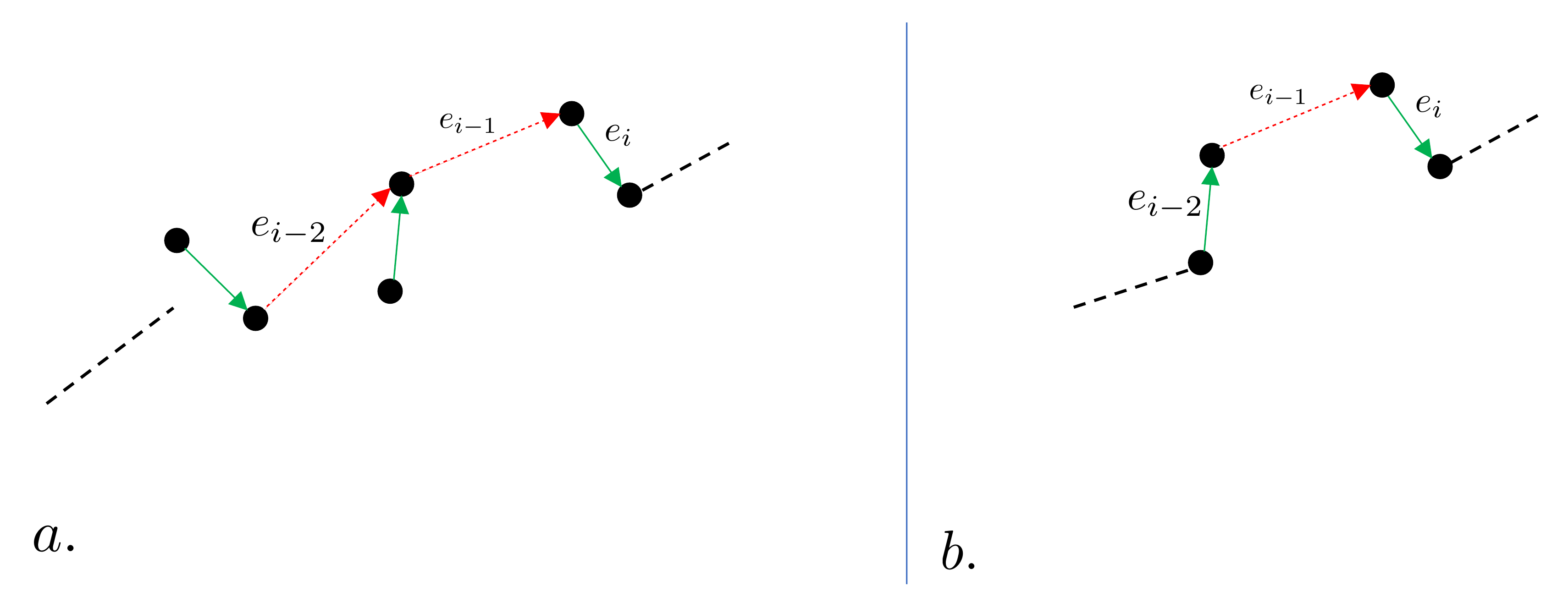}
\caption{ Different possibilities for $e_{i-1}$ and $e_{i-2}$ conditioned on the fact that $e_{i}$ is a $\mathfrak d$-type 
edge in the infinite path of the $f_{1}$-graph, if it exists. Solid (green) arrows represent $\mathfrak d$-type edges, whereas dashed (red) arrows represent $\delta$-type edges.}
 \label{F5}
\end{figure}
 \end{enumerate}
If $e_{i}$ is a  $\delta$-type edge, then there are three possibilities for $e_{i-1}$ and $e_{i-2}$:
\begin{enumerate}
\item both are $\delta$-type edges (see Figure \ref{F6}.a). In this case,  
$\|e_{i}\|< \max (\|e_{i-1}\|,\allowbreak \|e_{i-2}\|)$ since these edges are in a descending chain in \(\{S^0_i\}_i\).
\item $e_{i-1}$ is a $\mathfrak d$-type edge and $e_{i-2}$ is a $\delta$-type edge (see Figure \ref{F6}.b). Then
$\|e_{i-2}\| > \|e_{i}\|$ where the first inequality follows from the descending chain property. This implies that $\| e_{i}\|< \max (\|e_{i-1}\|, \|e_{i-2}\|)=\|e_{i-2}\|$. 
\item $e_{i-1}$ is a $\delta$-type edge and $e_{i-2}$ is a $\mathfrak d$-type edge 
(see Figure~\ref{F6}.c). 
The edges $e_{i-1}$ and $e_i$ are consecutive $\delta$-type edges in the descending chain on 
\(\{S_i^0\}_i\). Hence, $\|e_i\| < \|e_{i-1}\|$. Therefore, $\|e_i\| < \max\bigl(\|e_{i-1}\|,\|e_{i-2}\|\bigr).$
\end{enumerate}
This completes the proof of Proposition \ref{p005}.
\begin{figure}
\centering
\includegraphics[width=.9\textwidth]{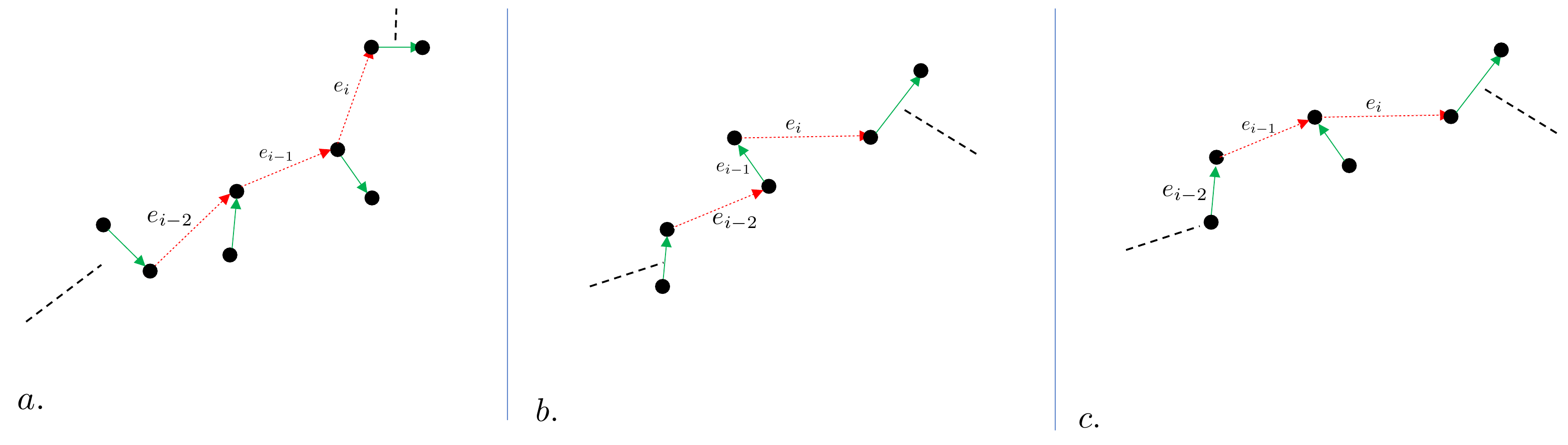}
\caption{ Different possibilities for $e_{i-1}$ and $e_{i-2}$ conditioned on the event that $e_{i}$ is a $\delta$-type edge in the infinite path of the $f_{1}$-graph if it exists. Solid (green) arrows represent $\mathfrak d$-type edges, whereas dashed (red) arrows represent $\delta$-type edges.}
 \label{F6}
\end{figure}
\end{proof}
Propositions \ref{004SOD} and  \ref{p005} give the following result. 
\begin{corollary}
\label{006}
If $\Phi^{0}$ is a Poisson point process, then  
the components of $f_{1}$-graph of $\Phi^{0}$ are all a.s. finite.
\end{corollary}
So far, it has been shown that there is no infinite path or cluster in the $f_{1}$-graph. 
The following proposition establishes that this holds true for all $f_{n}$-graphs.
\begin{proposition}
\label{008}
For all $n \in \mathbb{N}$, any infinite path in $f_{n}$-graph is a second-order descending chain. As a result the components of $f_{n}$-graphs of $\Phi^{0}$,for all  $n \in \mathbb{N}$, are all a.s finite.
\end{proposition} 

\begin{proof}
The proof is by induction. Any infinite path in the $f_{0}$-graph is a descending chain in the PPP. 
Therefore, the base case of the induction is true. For the induction step, suppose there is no infinite path 
(cluster) in the $f_n$-graph. Then, if an infinite path $P^{n+1}$ exists in the $f_{n+1}$-graph, 
this path lies on \(\{S^{n+1}_i\}_i\) eventually. Thus, there are two types of edges on $P^{n}$: 
the first type is that between two ${\mathrm{MNNn}}$ points, and the second type is that of edges that 
connect an exit point of level $n$, $x$, to ${\mathrm {NN(n+1)}}(x)$. The construction of the $f_{n+1}$ point-shift
shows that $P^{n+1}$ forms a descending chain between the $f_{n}$-cycles $\{S_{n}^{i}\}$. 
The same proof as in Proposition \ref{p005} shows that if $P^n$ exists, it is a second-order descending chain.
\end{proof}
\section{Related Work and Open Problems}
\label{OpenProblems}
We first briefly situate the results of the present paper among related spanning structures.

\paragraph{Factor graphs on PPP}
Ferrari, Landim and Thorisson~\cite{Hermann} gave a deterministic (factor) construction on a $d$-dimensional stationary PPP producing a loopless graph, a tree in dimensions $d=2,3$, with one end. Holroyd and Peres~\cite{Matching} further showed that in any dimension a factor of the PPP can be a one-ended tree. A \emph{factor} here is a graph whose vertex set is the support of the point process and whose edge set is a deterministic, isomorphism-invariant function of the configuration. In this sense, the $\mathrm{CHN}^2$ EFF on the PPP is also a factor graph (constructed level by level), and it is a one-ended forest.
Holroyd and Peres~\cite{Matching} showed that, in any dimension, the PPP admits a one-ended tree factor. Their construction is also based on a hierarchical sequence of finite partitions, although it does not arise from an explicit clustering algorithm. While their hierarchy is conceptually close to ours, their argument is analytical and does not prescribe an explicit algorithm for forming the components at each level.
Timar~\cite{timar2009} extended existence of tree factors beyond PPP to general stationary non-equidistant point processes.

\paragraph{Random thinning and hierarchical heads}
A thinning-based hierarchical classifier was proposed in~\cite{zuyev_tchoumatchenko2001}, where points are assigned random marks and attached to higher-level heads. Unlike that model, \(\mathrm{CHN}^2\) is an intrinsic deterministic factor of the underlying point process.

We conclude with a list of questions motivated by the $\mathrm{CHN}^2$ point-shift and its EFF.

\paragraph{Tree conjecture on PPP} Is the $\mathrm{CHN}^2$ EFF a tree (i.e., acyclic and connected componentwise) on the PPP in every dimension? Current results establish one-endedness of components; ruling out cycles at all scales remains open.

\paragraph{Point-map probability for $\mathrm{CHN}^2$} 
  For a point-shift $f$ acting on configurations, consider the semigroup of translations induced by $-f$ on probability measures supported on configurations with a point at the origin. The \emph{$f$-probability} (point-map probability) is the limit of the orbit of this action applied to the Palm distribution, when it exists (see, e.g., \cite{Pointshift}). Does the point-map probability exist for the $\mathrm{CHN}^2$ point-shift on the PPP? If so, can it be computed explicitly?

\paragraph{Alternative inter-cluster distances}
  The current $\mathrm{CHN}^2$ hierarchy uses the clustroid distance in its merging rule. Analyze the variants obtained by replacing this distance with, for example, Hausdorff distance or single-linkage distance between clusters. For these variants, determine:
  (i) finiteness of connected components at each hierarchical level;
  (ii) existence and identification of the local weak limit under the Palm probability;
   (iii) foil classification (e.g., $\mathcal{F/F}$ at finite levels and $\mathcal{I/I}$ in the limit) and number of ends of components.

\paragraph{Beyond PPP}
  On stationary Cox or, more generally, stationary non-equidistant point processes (cf.~\cite{timar2009}), to what extent do the PPP results persist? \\

\noindent
\textbf{Acknowledgements}
This work was supported by the ERC NEMO grant, under the European Union's Horizon 2020 research and innovation programme,
grant agreement number 788851 to INRIA.

\bibliographystyle{elsarticle-num}

\bibliography{references}

@article{Aldous,
	author = {Aldous, D. and Lyons, R.},
	journal = {Electronic Journal of Probability},
	pages = {1454--1508},
	title = {Processes on Unimodular Random Networks},
	volume = {12},
	year = {2007}}

@article{Baccelli2017,
	author = {Baccelli, F. and Haji-Mirsadeghi, M.-O. and Khezeli, A.},
	journal = {Contemporary Mathematics},
	month = {01},
	pages = {85-127},
	title = {Eternal Family Trees and Dynamics on Unimodular Random Graphs},
	year = {2018}}

@article{Murtagh,
	author = {Fionn Murtagh},
	journal = {The Computer Journal},
	pages = {354-359},
	title = {A survey of recent advances in hierarchical clustering algorithms},
	year = {1983}}

@article{Murtagh2,
	author = {Murtagh, F. and Contreras, P.},
	journal = {WIREs Data Mining and Knowledge Discovery},
	number = {6},
	title = {Algorithms for hierarchical clustering: an overview, II},
	volume = {7},
	year = {2017}}

@article{haji_mirsadeghi_baccelli2018,
  author = {Baccelli, F. and Haji-Mirsadeghi, M.-O.},
  title = {Point-Shift Foliation of a Point Process},
  journal = {Electronic Journal of Probability},
  year = {2018},
  volume = {23},
  pages = {1--25},
}

@book{Bartek,
	author = {Baccelli, Fran{\c c}ois and Blaszczyszyn, Bartlomiej and Karray, Mohamed},
	publisher = {INRIA},
	title = {Random Measures, Point Processes, and Stochastic Geometry},
	year = {2020}}

@article{Meester,
  author  = {H{\"a}ggstr{\"o}m, Olle and Meester, Ronald},
  title   = {Nearest neighbor and hard sphere models in continuum percolation},
  journal = {Random Structures \& Algorithms},
  year    = {1996},
  volume  = {9},
  pages   = {295--315}
}

@article{Matching,
	author = {Alexander Holroyd and Yuval Peres},
	journal = {Electronic Communications in Probability},
	pages = {17--27},
	title = {Trees and Matchings from Point Processes},
	volume = {8},
	year = {2003}}

@article{Daley,
	author = {Daryl J. Daley and G{\"u}nter Last},
	journal = {Advances in Applied Probability},
	month = {9},
	number = {3},
	pages = {604-628},
	publisher = {Advances in Applied Probability},
	title = {Descending Chains, the Lilypond Model, and Mutual-Nearest-Neighbour Matching},
	volume = {37},
	year = {2005}}

@book{Tho00,
	author = {Hermann Thorisson},
	publisher = {Springer},
	title = {Coupling, Stationarity, and Regeneration},
	year = {2000}}

@article{timar2009,
  author = {Adam Timar},
  title = {Tree and Grid Factors of General Point Processes},
  journal = {Electronic Communications in Probability},
  year = {2009},
  volume = {9},
  month = {September},
}

@article{zuyev_tchoumatchenko2001,
  author = {Sergei Zuyev and Konstantin Tchoumatchenko},
  title = {Aggregate and Fractal Tessellations},
  journal = {Probability Theory and Related Fields},
  publisher = {Springer},
  year = {2001},
  volume = {121},
  pages = {198--218},
}

@article{Hermann,
  author = {P.A. Ferrari and C.Landim  and H.Thorisson},
  journal = {Ann. I. H. Poincare},
  title = {Poisson trees, succession lines and coalescing random walks},
pages = {141-152},
  year = {2004}
}

@article{Mecke,
	author = {J. Mecke},
	journal = {Math. Nachrichten},
	pages = {335-344},
	title =  {Invarianzeigenschaften allgemeiner {P}almscher {M}asse},
	year = {1975}}

@book{1,
  author    = {Allan Gut},
  title     = {Stopped Random Walks: Limit Theorems and Applications},
  publisher = {Springer},
  year      = {1998}
}

@article{Pointshift,
  author  = {Baccelli, Fran{\c c}ois and Haji-Mirsadeghi, Mir-Omid},
  title   = {Point-map probabilities of a point process and Mecke’s invariant measure equation},
  journal = {The Annals of Probability},
  year    = {2017},
  volume  = {45},
  pages   = {1723--1751}
}
\end{document}